\documentclass[3p,times,sort&compress]{elsarticle}

\usepackage[T1]{fontenc}

\usepackage{amsmath,amssymb,amsthm,bm}

\usepackage{mathrsfs}

\usepackage{booktabs,array,float,microtype}

\usepackage{graphicx,tikz}

\usepackage[colorlinks=true,allcolors=blue]{hyperref}
\usepackage{orcidlink}

\theoremstyle{plain}
\newtheorem{theorem}{Theorem}[section]
\newtheorem{proposition}[theorem]{Proposition}
\newtheorem{corollary}[theorem]{Corollary}

\theoremstyle{remark}
\newtheorem{remark}[theorem]{Remark}
\theoremstyle{definition}
\newtheorem{example}[theorem]{Example}

\newcommand{\ZZ}{\mathbb{Z}}

\newcommand{\CC}{\mathbb{C}}

\newcommand{\PP}{\mathbb{P}}
\newcommand{\EE}{\mathbb{E}}

\newcommand{\mku}{\mkern1mu}
\renewcommand{\,}{\ifmmode\mkern2mu\else\thinspace\fi}

\newcommand{\deq}{\mathrel{\mathop:}=} 
\newcommand{\eqd}{=\mathrel{\mathop:}} 

\renewcommand{\leq}{\leqslant}
\renewcommand{\geq}{\geqslant}

\newcommand{\C}[2]{\binom{#1}{#2}}

\newcommand{\bis}[3]{\genfrac{(}{)}{0pt}{}{#1}{#2}_{\!#3}}

\newcommand{\bisnom}{\texorpdfstring{bi\textsuperscript{$s$}nomial}{bisnomial}}

\begin{document}

\begin{frontmatter}

\title{Bounded compositions and products of chains, or \bisnom{}s everywhere \\ from de~Moivre's formula to spin multiplicities and constant-sum codes}

\author[usthb]{Hac\`{e}ne Belbachir\,\orcidlink{0000-0001-8540-3033}}
\ead{hacenebelbachir@gmail.com} 
\affiliation[usthb]{organization={USTHB, Faculty of Mathematics, RECITS Laboratory}, addressline={P.O. Box 32 El Alia}, city={Bab Ezzouar, Algiers}, postcode={16111}, country={Algeria}}
\cortext[cor1]{Corresponding author.}

\author[usp]{Jos\'{e} Ricardo G. Mendon\c{c}a\corref{cor1}\,\orcidlink{0000-0002-5516-0568}}
\ead{jricardo@usp.br}
\affiliation[usp]{organization={Escola de Artes, Ci\^{e}ncias e Humanidades, Universidade de S\~{a}o Paulo}, addressline={Avenida Arlindo Béttio, 1000}, city={03828-000 S\~{a}o Paulo}, state={SP}, country={Brazil}} 

\begin{abstract}
How many ways are there to select $k$ objects from $n$ types when type $i$ is available in only $s_{i}$ copies? For equal bounds, this is de~Moivre's problem of obtaining a prescribed total with $n$ dice, now a textbook application of the principle of inclusion--exclusion. For arbitrary bounds, the answer is encoded by the rank polynomial $\prod_{i=1}^{n}(1+x+\cdots+x^{s_{i}})$. Its coefficients enumerate bounded compositions and submultisets, give, after normalization, the probability mass function of a sum of independent discrete uniform random variables, and are the weight multiplicities of a tensor product of $\mathrm{SU}(2)$ representations. Most of these interpretations are classical but dispersed; here we place them in a common notation and document their provenance. A complement involution proves symmetry and boundary identities. For arbitrary bounds, lower-half first differences give both symmetric-chain starts and mixed-spin multiplicities, and the chain-length and irreducible-dimension profiles agree. Convolution gives log-concavity and the central mode. For equal bounds with fixed $s \geq 1$ and even $ns$, local expansions give central asymptotics with explicit $1/n$ and $1/n^{2}$ corrections for the coefficient and its first difference. A Poisson limit gives the first boundary crossover of these differences at scale $k \sim \lambda n^{s/(s+1)}$, with factor $e^{-\alpha\lambda^{s+1}}$ when the proportion of bounds equal to the fixed minimum $s \geq 1$ tends to $\alpha>0$. Central first differences include the $s$-Catalan numbers, themselves Littlewood--Richardson coefficients; central layers yield maximum-size constant-sum codebooks for hash-based signatures. Under the polynomial binomial convention, the full inclusion--exclusion transform vanishes, and the omitted terms sum to the negative of the number of bounded compositions.
\end{abstract}


\end{frontmatter}


\section{Introduction}
\label{sec:intro}

Let $\bm{s}=(s_{1},\dots,s_{n})$ be a vector of non-negative integers and put $N=s_{1}+\cdots+s_{n}$. For $0 \leq k \leq N$, let $R_{\bm{s}}(k)$ be the number of integer solutions of
\begin{equation}
\label{eq:problem}
k_{1}+\cdots+k_{n} = k, \quad 0 \leq k_{i} \leq s_{i}, \quad i = 1,\dots,n,
\end{equation}
or, equivalently,
\begin{equation}
\label{eq:rankpoly}
R_{\bm{s}}(k) = [x^{k}]\prod_{i=1}^{n}(1+x+\cdots+x^{s_{i}}).
\end{equation}
Set $R_{\bm{s}}(k)=0$ for $k < 0$ or $k > N$. In the multiset model, type $i$ has $s_{i}$ indistinguishable objects, of which we choose $k_{i}$. Thus $R_{\bm{s}}(k)$ gives the number of size-$k$ submultisets of a multiset of total size $N$.

The solution to Problem~\eqref{eq:problem} is a standard exercise in inclusion--exclusion; see, for example, Charalambides \cite[Theorem~4.3]{Charalambides2002}, Riordan \cite[Problem~1.7]{Riordan1958}, or West \cite[Example~4.1.6]{West2021}. The generating polynomial in \eqref{eq:rankpoly} is the rank polynomial of a product of chains and, after normalization, the probability generating function of a sum of independent discrete uniform random variables. When $s_{i}=2\sigma_{i}$, with each $\sigma_{i} \in \tfrac{1}{2}\ZZ_{\geq 0}$, the coefficients are the weight-space dimensions of a tensor product of $\mathrm{SU}(2)$ representations. The first differences on the lower half of the array give the multiplicities of the irreducible constituents of this tensor product.

In the homogeneous case $s_{i}=s$ with even $n=2\nu$, the central first difference is the $s$-Catalan number, which Linz identifies with a value of a Littlewood--Richardson stretching polynomial. For $s=1$ and $s=2$, the resulting numbers are, respectively, the Catalan and even-indexed Riordan numbers. Near the boundary, by contrast, the coefficients and their first differences collapse to simplex numbers. The central rank layers of the product of chains are its largest antichains and therefore yield maximum-size constant-sum codebooks in the Winternitz framework, with applications to SPHINCS$^{+}$ and XMSS \cite{ZhangCui2023}. Table~\ref{tab:lives} summarizes the principal interpretations and their specializations.

\begin{table}[htbp]
\centering
\caption{Several interpretations of the coefficient array $R_{\bm{s}}(k)$.}
\label{tab:lives}
\medskip
\begin{tabular}{@{}lcc@{}}
\toprule
Interpretation of $R_{\bm{s}}(k)$ & Specialization & Section(s) \\
\midrule
Bounded compositions; size-$k$ submultisets & general $\bm{s}$
 & {\S}\ref{sec:pie}--{\S}\ref{sec:sym} \\
Probability mass function of a discrete uniform sum (after normalization)
 & general $\bm{s}$
 & {\S}\ref{sec:prob}, {\S}\ref{sec:extensions} \\
Rank sizes of a product of chains $\prod_{i}[0,s_{i}]$ & general $\bm{s}$
 & {\S}\ref{sec:chains} \\
One-dimensional discrete box spline & general $\bm{s}$
 & {\S}\ref{sec:conventions} \\
Weight multiplicities in $\bigotimes_{i}V_{\sigma_{i}}$ & $s_{i}=2\sigma_{i}$
 & {\S}\ref{sec:spin} \\
Irreducible multiplicities (lower-half first differences) & $s_{i}=2\sigma_{i}$
 & {\S}\ref{sec:spin} \\
Central first differences: $s$-Catalan and Littlewood--Richardson coefficients
 & $s_{i}=s$, $n=2\nu$ & {\S}\ref{sec:scat} \\
Constant-sum antichains for hash-based signatures
 & $s_{i}=s$, $k = \lfloor ns/2\rfloor$ & {\S}\ref{sec:hash} \\
\bottomrule
\end{tabular}
\end{table}

Our objective in this work is to survey these classical identities and interpretations and to record their provenance and terminology across the relevant literatures. In our derivations we carefully distinguish the globally valid and support-sharp forms of de Moivre's sum and combine the sharp cutoff with complement symmetry. We also connect first differences with symmetric-chain starts and equal- and mixed-spin multiplicities, and trace the boundary formulas from their exact range to a Poisson crossover. In an algebraic coda, we examine the effect of using the polynomial convention for binomial coefficients with negative upper arguments. Examples throughout the article illustrate the main identities and show how the different interpretations apply in concrete cases.


\section{Notation, terminology, and history}
\label{sec:notation}

\subsection{Notation}

Throughout, $n$ denotes the number of types, cells, or particles, depending on the interpretation. The integer $k$ is the prescribed total, and $k_{i}$ is the contribution of the component indexed by $i$. The individual upper bounds are $s_{i}$; in the homogeneous case they all equal $s$, and $N = s_{1}+\cdots+s_{n}$ is the maximum possible total. We use $R_{\bm{s}}(k)$ for the inhomogeneous enumeration in \eqref{eq:rankpoly} and write $\bis{n}{k}{s}$ when all bounds equal $s$. In the representation-theoretic sections, $\sigma$ denotes a single-particle spin, $m$ the total magnetization, and $J$ the total spin; when the number of particles is even, we write $n=2\nu$. The symbol $q$ is reserved for deformation parameters.

The homogeneous notation $\bis{n}{k}{s}$ follows Belbachir and Benmezai \cite{BelbachirBenmezai2014} and Linz \cite{Linz2022}; the notation $R_{\bm{s}}(k)$ keeps the full vector of unequal bounds explicit.

\subsection{Terminology}

In the homogeneous case the numbers defined by~\eqref{eq:problem} are the \emph{\bisnom{} coefficients} $\bis{n}{k}{s}$, defined by the expansion
\begin{equation}
\label{eq:bisdef}
(1+x+x^{2}+\cdots+x^{s})^{n} = \sum_{k \geq 0}\bis{n}{k}{s}\,x^{k},
\quad \bis{n}{k}{s} = 0 \ \text{ for } k < 0 \text{ or } k > ns.
\end{equation}
Other names for these coefficients include \emph{ordinary multinomial coefficients}, \emph{polynomial coefficients}, and \emph{extended} or \emph{generalized binomial coefficients of order} $s+1$. Here \emph{polynomial coefficients} names the array, not the convention for ordinary binomial coefficients that we discuss in Section~\ref{sec:conventions}. For $s=1$, the coefficients reduce to the ordinary binomial coefficients, $\bis{n}{k}{1} = \C{n}{k}$. The full coefficient array is the \emph{\bisnom{} triangle}, also called the \emph{$s$-Pascal triangle}.

The abundance of names reflects the array's appearance in several literatures. Belbachir and Igueroufa \cite[Table~1]{BelbachirIgueroufa2020b} provide a historical census, extending from Euler (1778) to Belbachir and Benmezai (2014), and date the term \emph{\bisnom{}} to the latter authors \cite{BelbachirBenmezai2014}. Table~2 of the same paper catalogs the names used for the coefficient array.

\subsection{History}

The homogeneous form of Problem~\eqref{eq:problem} goes back to de Moivre. In its original guise, it asks for the number of ways to obtain a prescribed total with $n$ dice having $s+1$ numbered faces. De Moivre stated the problem for general $n$ and $s$ without proof in \emph{De Mensura Sortis} (1711). His first published proof, based essentially on generating functions, appeared in the \emph{Miscellanea Analytica} (1730, p.~196). Montmort had independently obtained an inclusion--exclusion solution by 1713; see Hald \cite{Hald1984} and Doumas \cite{Doumas2013}. The gap between de Moivre's statement and his first published proof explains some of the conflicting dates in the later coefficient literature.

Freund \cite{Freund1956} later interpreted $\bis{n}{k}{s}$ as the number of distributions of $k$ objects among $n$ cells with at most $s$ objects in each cell. He treated the problem under the heading of \emph{restricted occupancy theory}; see also Bondarenko \cite{Bondarenko1993}.


\section{Enumeration by inclusion--exclusion}
\label{sec:pie}

\subsection{Two basic formulas}

The enumeration in Problem~\eqref{eq:problem} uses two basic combinatorial formulas. For subsets $A_{1},\dots,A_{M}$ of a finite set $S$, the first is the principle of inclusion--exclusion in the form
\begin{equation}
\label{eq:pie}
\Big|S \setminus \bigcup_{i=1}^{M}A_{i}\Big| \,= 
\sum_{I \subseteq \{1,\dots,M\}}(-1)^{|I|}\,\Big|\bigcap_{i\in I}A_{i}\Big|,
\end{equation}
where the intersection indexed by the empty set is $S$. The second is the stars-and-bars formula for the number of non-negative solutions of $k_{1}+\cdots+k_{n}=k$, namely
\begin{equation}
\label{eq:stars}
\C{k+n-1}{n-1} = \C{k+n-1}{k},
\end{equation}
where the equality follows from the symmetry of the ordinary binomial coefficient.

By default, we use the \emph{elementary convention} for binomial coefficients
\begin{equation}
\label{eq:elementary}
\C{u}{v} = 0 ~\text{ unless }~ 0 \leq v \leq u.
\end{equation}
This convention is standard in the \bisnom{} literature \cite{BelbachirBouroubi2008,BelbachirIgueroufa2020b,BelbachirBenmezai2014}. It makes inclusion--exclusion terms with negative or too-small upper arguments vanish, as required for empty intersections. In Section~\ref{sec:conventions}, we show that the generating function produces exactly the terms the convention retains and evaluate the sum of the additional terms that the alternative \emph{polynomial convention} (see, for instance, \cite[{\S}5.1]{GrahamKnuth1994}) contributes. Under this latter convention, $\C{u}{v}$ is the polynomial $u(u-1)\cdots(u-v+1)/v!$ in $u$ for each fixed integer $v \geq 0$.

\subsection{Enumeration in the inhomogeneous case}

For $0 \leq k \leq N$, let $S$ be the set of non-negative integer solutions of $k_{1}+\cdots+k_{n}=k$, without the upper bounds in~\eqref{eq:problem}. The stars-and-bars formula gives
\begin{equation}
\label{eq:Sigma0}
\Sigma_{0} = |S| = \C{k+n-1}{n-1}.
\end{equation}
For each $i$, let $A_{i}\subseteq S$ be the set of solutions that violate the upper bound on $k_{i}$, so that $k_{i} \geq s_{i}+1$. Writing $k_{i} = s_{i}+1+b_{i}$ gives a bijection between $A_{i}$ and the non-negative solutions of
\begin{equation}
\label{eq:shift}
k_{1}+\cdots+b_{i}+\cdots+k_{n} = k-(s_{i}+1),
\end{equation}
whose number is, by~\eqref{eq:stars},
$\C{k-(s_{i}+1)+n-1}{n-1}$. Hence
\begin{equation}
\label{eq:Sigma1}
\Sigma_{1} = \sum_{i=1}^{n}|A_{i}| = \sum_{i=1}^{n}\C{k-(s_{i}+1)+n-1}{n-1}.
\end{equation}
For distinct indices $i_{1},\dots,i_{j}$, shifting the corresponding $j$ coordinates reduces the prescribed total by $\sum_{l=1}^{j}(s_{i_{l}}+1)$. Summing over all $j$-element sets of indices gives
\begin{equation}
\label{eq:Sigmaj}
\Sigma_{j} = \sum_{1 \leq i_{1} < \cdots < i_{j} \leq n}
|A_{i_{1}} \cap \cdots \cap A_{i_{j}}| = 
\sum_{1 \leq i_{1} < \cdots < i_{j} \leq n} \C{\,k-\sum\limits_{l=1}^{j}(s_{i_{l}}+1)+n-1\,}{n-1}.
\end{equation}
Because $k \leq N$, the intersection $A_{1} \cap \cdots \cap A_{n}$ is empty: if every upper bound were violated, then $\sum_{i=1}^{n}k_{i} \geq N+n > k$. Hence $\Sigma_{n}=0$, and inclusion--exclusion gives
\begin{equation}
\label{eq:9}
R_{\bm{s}}(k) = \sum_{j=0}^{n-1}(-1)^{j} \sum_{1 \leq i_{1} < \cdots < i_{j} \leq n} \C{\,k-\sum\limits_{l=1}^{j}(s_{i_{l}}+1)+n-1\,}{n-1}.
\end{equation}
For $j=0$, the inner sum contains only the empty index set and equals $\C{k+n-1}{n-1}$.

Formula~\eqref{eq:9} is a special case of the lower-and-upper-bound result of Charalambides \cite[Theorem~4.3]{Charalambides2002}. MacMahon \cite[{\S\S}38--39]{MacMahon1920} treats bounded multiplicities by symmetric-function methods.

If $L_{i}$ denotes the number of types whose bound equals $i$, collecting equal bounds rewrites the generating polynomial as $\prod_{i \geq 1}(1+x+\cdots+x^{i})^{L_{i}}$. The coefficients in this expansion are the classical supernomial coefficients of Schilling and Warnaar \cite{SchillingWarnaar1998}; see also Belbachir, Bouroubi, and Khelladi \cite[Section~1]{BelbachirBouroubi2008}. Equations~\eqref{eq:27}--\eqref{eq:28} describe the corresponding supernomial boundary regimes.

\begin{example}[Unequal bounds]
\label{ex:unequal-overlap}
Take $\bm{s}=(1,2,4)$ and $k=5$. There are $\C{7}{2}=21$ unrestricted non-negative solutions. The sets defined by $k_{1}\geq2$, $k_{2}\geq3$, and $k_{3}\geq5$ have sizes $\C{5}{2}=10$, $\C{4}{2}=6$, and $\C{2}{2}=1$, respectively. The only nonempty pairwise intersection is $A_{1}\cap A_{2}=\{(2,3,0)\}$, so equation~\eqref{eq:9} gives
\[
R_{(1,2,4)}(5)=21-10-6-1+1=5.
\]
The five admissible tuples are $(0,1,4)$, $(0,2,3)$, $(1,0,4)$, $(1,1,3)$, and $(1,2,2)$. The final $+1$ corrects the double subtraction of $(2,3,0)$.
\end{example}

\subsection{The homogeneous enumeration: de Moivre's alternating summation}

Suppose $n > 1$ and all bounds equal $s$, so that $N = ns$. Riordan \cite[Problem~1.7]{Riordan1958} calls the resulting homogeneous coefficients \emph{regular}. In~\eqref{eq:Sigmaj}, the binomial coefficient then depends only on $j$, while the $j$ restricted coordinates can be chosen in $\C{n}{j}$ ways. Hence
\begin{equation}
\label{eq:10}
\Sigma_{j} = \C{n}{j}{\mku}\C{k-j(s+1)+n-1}{n-1}.
\end{equation}
On the admissible range $0 \leq k \leq ns$, we have $\Sigma_{n}=0$, and inclusion--exclusion reduces to
\begin{equation}
\label{eq:11}
\bis{n}{k}{s} = \sum_{j=0}^{n-1}(-1)^{j}\Sigma_{j} = \sum_{j=0}^{n-1}(-1)^{j}{\mku}\C{n}{j}{\mku}\C{k-j(s+1)+n-1}{n-1}.
\end{equation}
For $(n,s,k)=(3,2,4)$, equation~\eqref{eq:11} gives $\C{6}{2}-3\C{3}{2}=6$, corresponding to the three permutations of $(2,2,0)$ and the three permutations of $(2,1,1)$.

Identity~\eqref{eq:11} is de Moivre's alternating summation. Belbachir, Bouroubi, and Khelladi \cite[Theorem~2.1]{BelbachirBouroubi2008} state and prove the support-sharp form
\begin{equation}
\label{eq:BBK21}
\bis{n}{k}{s} = \sum_{j=0}^{\lfloor k/(s+1)\rfloor}(-1)^{j}{\mku}\C{n}{j}{\mku}\C{k-j(s+1)+n-1}{n-1}.
\end{equation}
Their proof uses the generating-function factorization
\begin{equation}
\label{eq:gfproof}
(1+x+\cdots+x^{s})^{n} = \frac{(1-x^{s+1})^{n}}{(1-x)^{n}} = \Bigg[\sum_{j=0}^{n}(-1)^{j}{\mku}\C{n}{j}\,x^{j(s+1)}\Bigg] \Bigg[\sum_{i \geq 0}\C{i+n-1}{n-1}\,x^{i}\Bigg].
\end{equation}
Extracting the coefficient of $x^{k}$ from~\eqref{eq:gfproof} gives~\eqref{eq:BBK21}. Belbachir and Benmezai \cite[Eq.~(5)]{BelbachirBenmezai2014} and Belbachir and Igueroufa \cite[Eq.~(3)]{BelbachirIgueroufa2020b} also give~\eqref{eq:BBK21} under the name \emph{de Moivre alternate summation}. Thus \eqref{eq:BBK21} is the support-sharp form of \eqref{eq:11}.

\begin{remark}[Presentations in the literature]
\label{rem:presentations}
Equation~\eqref{eq:11} occurs in several equivalent forms, differing in the form of the binomial coefficient and the upper limit of summation:
\medskip
\begin{center}
\small
\begin{tabular}{@{} > {\raggedright\arraybackslash}p{0.40\linewidth}
 > {\centering\arraybackslash}p{0.36\linewidth}
 > {\centering\arraybackslash}p{0.15\linewidth}@{}}
\toprule
Source & Binomial coefficient & Upper limit on $j$ \\
\midrule 
Riordan \cite[Prob.~1.7]{Riordan1958}
 & $\displaystyle\C{k-j(s+1)+n-1}{k-j(s+1)}$ & implicit \\[3ex]
West \cite[Ex.~4.1.6]{West2021}, Charalambides \cite[Ex.~4.7]{Charalambides2002}
 & $\displaystyle\C{k-j(s+1)+n-1}{n-1}$ & $n$ \\[3ex]
Charalambides \cite[eq.~(2.8)]{Charalambides1982}
 & $\displaystyle\C{k-j(s+1)+n-1}{n-1}$ & $\lfloor k/(s+1)\rfloor$ \\[3ex]
Belbachir--Bouroubi--Khelladi \cite[Th.~2.1]{BelbachirBouroubi2008}
 & $\displaystyle\C{k-j(s+1)+n-1}{n-1}$ & $\lfloor k/(s+1)\rfloor$ \\[3ex]
Equation~\eqref{eq:11} above
 & $\displaystyle\C{k-j(s+1)+n-1}{n-1}$ & $n-1$ \\
\bottomrule
\end{tabular}
\end{center}

\medskip

Riordan obtains his binomial coefficient from the symmetry $\C{u}{v}=\C{u}{u-v}$ for non-negative arguments. Terms whose lower argument is negative vanish under the elementary convention, so his form requires no explicit cutoff. West \cite{West2021} and Charalambides \cite{Charalambides2002} sum through $j=n$, giving the full generating-function expansion, which is valid for every $k \geq 0$. When $0 \leq k \leq ns$, the $j=n$ term vanishes and all the forms in the table agree. Thus the cutoff $n-1$ in \eqref{eq:11} is valid only on the admissible range. Charalambides gives the support-sharp cutoff in his 1982 paper \cite{Charalambides1982} and the globally valid but generally unsharp cutoff $n$ in his 2002 textbook \cite{Charalambides2002}.
\end{remark}

\begin{proposition}[Exact summation bounds]
\label{prop:bound}
For every $k \geq 0$, the nonzero terms in the full generating-function expansion are exactly those with
\begin{equation}
0 \leq j \leq \min\bigg\{n,\bigg\lfloor\frac{k}{s+1}\bigg\rfloor\bigg\}.
\end{equation}
Consequently, the globally valid coefficient formula is
\begin{equation}
\label{eq:globalsharp}
\bis{n}{k}{s} = \sum_{j=0}^{\min\{n,\lfloor k/(s+1)\rfloor\}}
(-1)^{j}{\mku}\C{n}{j}{\mku}\C{k-j(s+1)+n-1}{n-1},
\end{equation}
where the left-hand side is zero for $k > ns$. On the admissible range $0 \leq k \leq ns$, the minimum is unnecessary; Belbachir, Bouroubi, and Khelladi \cite[Theorem~2.1]{BelbachirBouroubi2008} and Charalambides \cite[eq.~(2.8)]{Charalambides1982} give the sharp upper limit for nonzero summands $\lfloor k/(s+1)\rfloor$.
\end{proposition}

\begin{proof}
The factor $\C{n}{j}$ is nonzero exactly for $0 \leq j \leq n$, whereas the second binomial coefficient is nonzero exactly when $j \leq k/(s+1)$. A summand is nonzero precisely when both conditions hold, which gives the range stated in the proposition. Identity~\eqref{eq:globalsharp} then follows by coefficient extraction from \eqref{eq:gfproof}. If $0 \leq k \leq ns$, then $k/(s+1) \leq ns/(s+1) < n$, so the minimum reduces to $\lfloor k/(s+1)\rfloor$ and every displayed term is nonzero.

\end{proof}

\begin{example}[Beyond the support]
\label{ex:off-support}
The distinction matters off the admissible range. For $s=1$, $n=2$, and $k=4$, the globally sharp cutoff is $2$ and
\begin{equation}
\C{2}{0}\C{5}{1}-\C{2}{1}\C{3}{1}+\C{2}{2}\C{1}{1} = 5-6+1 = 0,
\end{equation}
whereas truncating at $j=n-1=1$ gives $-1$.
\end{example}

Table~\ref{tab:bisnomial-triangles} displays the first rows of the coefficient triangles for $s=2,3$; the coefficient triangle for $s=1$ is just the usual binomial (Pascal) triangle. Read by rows, with $n \geq 0$ and $0 \leq k \leq ns$, these arrays are catalogued in the On-Line Encyclopedia of Integer Sequences (OEIS) as sequences \href{https://oeis.org/A027907}{A027907} and \href{https://oeis.org/A008287}{A008287}, respectively; the corresponding array for $s=4$ is \href{https://oeis.org/A035343}{A035343}.

\begin{table}[H]
\centering
\caption{The \bisnom{} coefficient triangles for $s=2,3$ and $0 \leq n \leq 6$. Each row indexed by $n$ lists $\bis{n}{k}{s}$; blank entries are zero. The central coefficient(s) at $k=\lfloor ns/2\rfloor$ and $k=\lceil ns/2\rceil$ are in bold, and the last column gives the row sum $(s+1)^n$.} 
\label{tab:bisnomial-triangles}
\medskip
\begingroup
\small
\setlength{\tabcolsep}{1.5pt}
\textbf{(a) $s=2$: trinomial coefficients} \\[3pt]
\begin{tabular}{@{}c@{\hspace{1em}}*{13}{>{\centering\arraybackslash}p{2em}}@{\hspace{1em}}c@{}}
\toprule
$n\backslash k$ & 0 & 1 & 2 & 3 & 4 & 5 & 6 & 7 & 8 & 9 & 10 & 11 & 12 & $(s+1)^n$ \\
\midrule
0 & \textbf{1} &  &  &  &  &  &  &  &  &  &  &  &  & 1 \\
1 & 1 & \textbf{1} & 1 &  &  &  &  &  &  &  &  &  &  & 3 \\
2 & 1 & 2 & \textbf{3} & 2 & 1 &  &  &  &  &  &  &  &  & 9 \\
3 & 1 & 3 & 6 & \textbf{7} & 6 & 3 & 1 &  &  &  &  &  &  & 27 \\
4 & 1 & 4 & 10 & 16 & \textbf{19} & 16 & 10 & 4 & 1 &  &  &  &  & 81 \\
5 & 1 & 5 & 15 & 30 & 45 & \textbf{51} & 45 & 30 & 15 & 5 & 1 &  &  & 243 \\
6 & 1 & 6 & 21 & 50 & 90 & 126 & \textbf{141} & 126 & 90 & 50 & 21 & 6 & 1 & 729 \\
\bottomrule
\end{tabular}

\bigskip

\textbf{(b) $s=3$: quadrinomial coefficients} \\[3pt]
\begin{tabular}{@{}c@{\hspace{1em}}*{19}{>{\centering\arraybackslash}p{2em}}@{\hspace{1em}}c@{}}
\toprule
$n\backslash k$ & 0 & 1 & 2 & 3 & 4 & 5 & 6 & 7 & 8 & 9 & 10 & 11 & 12 & 13 & 14 & 15 & 16 & 17 & 18 & $(s+1)^n$ \\
\midrule
0 & \textbf{1} &  &  &  &  &  &  &  &  &  &  &  &  &  &  &  &  &  &  & 1 \\
1 & 1 & \textbf{1} & \textbf{1} & 1 &  &  &  &  &  &  &  &  &  &  &  &  &  &  &  & 4 \\
2 & 1 & 2 & 3 & \textbf{4} & 3 & 2 & 1 &  &  &  &  &  &  &  &  &  &  &  &  & 16 \\
3 & 1 & 3 & 6 & 10 & \textbf{12} & \textbf{12} & 10 & 6 & 3 & 1 &  &  &  &  &  &  &  &  &  & 64 \\
4 & 1 & 4 & 10 & 20 & 31 & 40 & \textbf{44} & 40 & 31 & 20 & 10 & 4 & 1 &  &  &  &  &  &  & 256 \\
5 & 1 & 5 & 15 & 35 & 65 & 101 & 135 & \textbf{155} & \textbf{155} & 135 & 101 & 65 & 35 & 15 & 5 & 1 &  &  &  & 1024 \\
6 & 1 & 6 & 21 & 56 & 120 & 216 & 336 & 456 & 546 & \textbf{580} & 546 & 456 & 336 & 216 & 120 & 56 & 21 & 6 & 1 & 4096 \\
\bottomrule
\end{tabular}
\endgroup
\end{table}


\section{Symmetry, boundary collapse, and figurate numbers}
\label{sec:sym}

\subsection{The bijection and the symmetry relation}

Given a solution of Problem~\eqref{eq:problem}, set $b_{i} = s_{i}-k_{i}$. Then $0 \leq b_{i} \leq s_{i}$ for every $i$, and
\begin{equation}
\label{eq:15}
b_{1}+\cdots+b_{n} = \sum_{i=1}^{n}s_{i}-k = N-k.
\end{equation}
Conversely, $k_{i} = s_{i}-b_{i}$ recovers the original solution. Thus coordinatewise complementation replaces selected copies by omitted copies and gives a bijection between solutions of total $k$ and solutions of total $N-k$. In the homogeneous case, it proves the \emph{symmetry relation}
\begin{equation}
\label{eq:symmetry}
\bis{n}{k}{s} = \bis{n}{ns-k}{s}.
\end{equation}
Belbachir, Bouroubi, and Khelladi \cite[eq.~(1.3)]{BelbachirBouroubi2008} and Belbachir and Igueroufa \cite[eq.~(4)]{BelbachirIgueroufa2020b} record this identity.

The stars-and-bars formula gives $\C{N-k+n-1}{n-1}$ unrestricted non-negative solutions in the variables $b_{i}$. The upper bounds are automatic precisely in the following range.

\begin{proposition}
\label{prop:iff}
For $0 \leq k \leq N$, the identity $R_{\bm{s}}(k) = \C{N-k+n-1}{n-1}$ holds if and only if $N-k \leq \min(s_{1},\dots,s_{n})$.
\end{proposition}

\begin{proof}
By the complement bijection, it remains to determine when every non-negative solution of $b_{1}+\cdots+b_{n} = N-k$ automatically satisfies $b_{i} \leq s_{i}$. If $N-k \leq \min(s_{1},\dots,s_{n})$, then $b_{i} \leq N-k \leq s_{i}$ for every $i$. Conversely, choose $i^{*}$ such that $s_{i^{*}} = \min(s_{1},\dots,s_{n})$. If $N-k > s_{i^{*}}$, the unrestricted solution $b_{i^{*}} = N-k$, $b_{l} = 0$ for $l \neq i^{*}$ is inadmissible, so the stars-and-bars formula strictly overcounts.
\end{proof}

\subsection{Figurate numbers}

When $s_{1}=\cdots=s_{n}=s$, so that $N = ns$, put $c = ns-k$. For $0 \leq c \leq s$, Proposition~\ref{prop:iff} gives
\begin{equation}
\label{eq:19}
\bis{n}{ns-c}{s} = \C{c+n-1}{n-1} = \C{c+n-1}{c},
\end{equation}
where the second equality follows from the symmetry of the ordinary binomial coefficient. For fixed $c$, the numbers $\C{c+n-1}{c}$ form the $c$-dimensional simplex-number sequence --- triangular for $c=2$, tetrahedral for $c=3$, and so on. In the binomial (Pascal) coefficient triangle these figurate sequences lie on diagonals. Left-aligning the rows turns these diagonal figurate sequences into columns, as Table~\ref{tab:pascal-figurate} illustrates. Thus equation~\eqref{eq:19} identifies the coefficient at distance $c$ from the upper edge of the support with the corresponding figurate number.

\begin{table}[H]
\centering
\caption{A segment of the left-aligned binomial coefficient triangle. The entries below the diagonal label $c$ form the sequence $\C{c+n-1}{c}$, $n \geq 1$; in particular, the columns $c = 2$ and $c = 3$ contain the triangular and tetrahedral numbers.}
\label{tab:pascal-figurate}
\medskip
\setlength{\tabcolsep}{4pt}
\renewcommand{\arraystretch}{1.0}
\begin{tabular}{@{}cccccccc@{}}
\underline{$c=0$} & & & & & & & \\
$1$ & \underline{$c=1$} & & & & & & \\
$1$ & $1$ & \underline{$c=2$} & & & & & \\
$1$ & $2$ & $1$ & \underline{$c=3$} & & & & \\
$1$ & $3$ & $3$ & $1$ & \underline{$c=4$} & & & \\
$1$ & $4$ & $6$ & $4$ & $1$ & \underline{$c=5$} & & \\
$1$ & $5$ & $10$ & $10$ & $5$ & $1$ & \underline{$c=6$} & \\
$1$ & $6$ & $15$ & $20$ & $15$ & $6$ & $1$ & \\
$\vdots$ & $\vdots$ & $\vdots$ & $\vdots$ & $\vdots$ & $\vdots$ & $\vdots$ & $\ddots$ \\
\end{tabular}
\end{table}

\subsection{The boundary collapse, three times over}

Formula~\eqref{eq:11} collapses to a single binomial coefficient at either end of the admissible range of $k$. In the lower regime this is immediate from the elementary convention. For $0 \leq k \leq s$ and any $j \geq 1$ we have $j(s+1) \geq s+1 > k$, so $k-j(s+1)+n-1 < n-1$ and the corresponding binomial vanishes by~\eqref{eq:elementary}. Only the $j=0$ term survives, giving
\begin{equation}
\label{eq:22}
\sum_{j=0}^{n-1}(-1)^{j}{\mku}\C{n}{j}{\mku}\C{k-j(s+1)+n-1}{n-1} = \C{k+n-1}{n-1}, \quad 0 \leq k \leq s.
\end{equation}
In \bisnom{} notation, equation~\eqref{eq:22} says that $\bis{n}{k}{s} = \C{n+k-1}{k}$ for $0 \leq k \leq s$.

Belbachir, Bouroubi, and Khelladi \cite[Corollary~4.2]{BelbachirBouroubi2008} give~\eqref{eq:22}; Charalambides \cite[Example~4.7]{Charalambides2002} records the same boundary case. The Bell-polynomial and complement-bijection arguments below emphasize different aspects of the identity.

\paragraph{Bell-polynomial proof}
Let $B_{\nu,n}(t_{1},t_{2},\dots)$ be the exponential partial Bell polynomials, defined by
\begin{equation}
\label{eq:belldef}
\frac{1}{n!}\bigg(\sum_{\ell \geq 1}\frac{t_{\ell}}{\ell!}x^{\ell}\bigg)^{n} = \sum_{\nu \geq n}B_{\nu,n}\frac{x^{\nu}}{\nu!}, \quad n = 0,1,2,\dots
\end{equation}
Taking $t_{\ell}=\ell!$ for $1 \leq \ell \leq s+1$ and $t_{\ell}=0$ otherwise turns the inner sum in~\eqref{eq:belldef} into $x+x^{2}+\cdots+x^{s+1} = x(1+x+\cdots+x^{s})$, and comparing coefficients of $x^{\nu}$ yields the connection formula
\cite[Theorem~4.1]{BelbachirBouroubi2008}
\begin{equation}
\label{eq:bellconn}
B_{\nu,n}\big(1!,2!,\dots,(s+1)!,0,\dots\big) = \frac{\nu!}{n!}\mku\bis{n}{\nu-n}{s}.
\end{equation}
With no truncation, the inner sum is $x/(1-x)$, and the same comparison gives, for $\nu \geq n \geq 1$, the classical evaluation by Lah numbers,
\begin{equation}
B_{\nu,n}(1!,2!,3!,\dots) = \frac{\nu!}{n!}\C{\nu-1}{\nu-n}.
\end{equation}
The coefficient of $x^{\nu}$ in~\eqref{eq:belldef} involves only the products $t_{\ell_{1}}\cdots t_{\ell_{n}}$ with $\ell_{1}+\cdots+\ell_{n} = \nu$ and every $\ell_{i} \geq 1$, so that $\ell_{i} \leq \nu-n+1$. Consequently, truncating the Bell-polynomial arguments after $(s+1)!$ has no effect when $s+1 \geq \nu-n+1$. For $s \geq \nu-n$, the two evaluations therefore agree and $\bis{n}{\nu-n}{s} = \C{\nu-1}{\nu-n}$, which is~\eqref{eq:22} on setting $k = \nu-n$.

\paragraph{Bijection proof}
Apply Proposition~\ref{prop:iff} to the complementary total $ns-k$. Since $ns-(ns-k)=k \leq s$, the proposition gives $\bis{n}{ns-k}{s} = \C{k+n-1}{n-1}$; the symmetry relation~\eqref{eq:symmetry} then yields~\eqref{eq:22}.

The symmetry relation reflects the lower boundary onto the upper boundary. Indeed, $ns-k \in [0,s]$ precisely when $k \in [(n-1)s,ns]$. Applying~\eqref{eq:22} at $ns-k$ therefore gives $\bis{n}{k}{s} = \C{ns-k+n-1}{n-1}$, that is,
\begin{equation}
\label{eq:24}
\sum_{j=0}^{n-1}(-1)^{j}{\mku}\C{n}{j}{\mku}\C{k-j(s+1)+n-1}{n-1} = \C{ns-k+n-1}{n-1}, \quad (n-1)s \leq k \leq ns.
\end{equation}
In~\eqref{eq:22}, every term but one vanishes individually; in \eqref{eq:24}, several terms cancel to leave the single figurate number on the right.

\begin{example}[Three dice near the upper boundary]
\label{ex:dice-boundary}
For three ordinary dice totaling $16$, subtracting one from each face value gives $(n,s,k)=(3,5,13)$. The support-sharp inclusion--exclusion formula yields
\[
\bis{3}{13}{5}=\C{15}{2}-3\C{9}{2}+3\C{3}{2}=105-108+9=6.
\]
Complement symmetry replaces $k=13$ by $15-13=2$, where the lower-boundary formula gives $\bis{3}{2}{5}=\C{4}{2}=6$ directly. The six dice outcomes are the three permutations of $(6,6,4)$ and the three permutations of $(6,5,5)$. On the original dice, complementation sends each face value $d_{i}$ to $7-d_{i}$, pairing totals $16$ and $5$; each total has probability $6/6^{3}=1/36$.
\end{example}

\subsection{Two consequences of the symmetry relation}

At $k = ns+1$, no bounded solution exists. The globally valid generating-function expansion obtained from~\eqref{eq:gfproof} therefore sums to zero. For $n \geq 2$, its $j=n$ term is $(-1)^{n}\C{n}{n}\C{0}{n-1} = 0$, so the sum may be written through $j=n-1$:
\begin{equation}
\label{eq:13}
\sum_{j=0}^{n-1}(-1)^{j}{\mku}\C{n}{j}{\mku}\C{(n-j)(s+1)}{n-1} = 0,
\end{equation}
independently of $s$. At the nearby admissible value $k = ns-1$, for $s \geq 1$, exactly $n$ solutions remain --- one for each choice of which $k_{i}$ drops to $s-1$ while the rest stay at $s$ --- and hence
\begin{equation}
\label{eq:14}
\sum_{j=0}^{n-1}(-1)^{j}{\mku}\C{n}{j}{\mku}\C{(n-j)(s+1)-2}{n-1} = n,
\end{equation}
again independently of $s$. West \cite[Example~4.1.6]{West2021} derives both identities by the same inclusion--exclusion argument.

\begin{remark}
\label{rem:sym}
With the coefficient array extended by zero outside its support as in~\eqref{eq:bisdef}, the symmetry relation holds for every integer $k$. It gives $\bis{n}{ns+1}{s} = \bis{n}{-1}{s} = 0$ and, for $s \geq 1$, $\bis{n}{ns-1}{s} = \bis{n}{1}{s} = n$. Thus equations~\eqref{eq:13} and~\eqref{eq:14} are the alternating-sum forms of these two elementary values. Each left-hand side is an alternating sum of $n$ degree-$(n-1)$ polynomials in $s$, yet its value is the constant $0$ or $n$. Equation~\eqref{eq:24} similarly results from composing~\eqref{eq:22} with the symmetry relation.
\end{remark}

Complement symmetry also shortens the support-sharp inclusion--exclusion sum.

\begin{corollary}
\label{cor:sharpsym}
For $0 \leq k \leq ns$, let $k^{*} = \min(k,ns-k)$ denote the distance from $k$ to the nearer endpoint of the support. Then
\begin{equation}
\label{eq:sharpsym}
\bis{n}{k}{s} = \sum_{j=0}^{\lfloor k^{*}/(s+1)\rfloor}(-1)^{j}{\mku}\C{n}{j}{\mku}\C{k^{*}-j(s+1)+n-1}{n-1},
\end{equation}
a sum of $\lfloor k^{*}/(s+1)\rfloor+1$ terms.
\end{corollary}

\begin{proof}
Apply~\eqref{eq:BBK21} at $k^{*}$ and use~\eqref{eq:symmetry}.
\end{proof}

Since $k^{*} \leq ns/2$, the number of terms in \eqref{eq:sharpsym} never exceeds $\lfloor ns/2(s+1)\rfloor+1$, roughly half the maximum number of terms in \eqref{eq:BBK21}. At $k = ns$, equation~\eqref{eq:sharpsym} has one term, whereas the direct support-sharp form has $\lfloor ns/(s+1)\rfloor+1$ terms. Finally, $\lfloor k^{*}/(s+1)\rfloor = 0$ exactly when $k^{*} \leq s$. This condition describes the lower range $0 \leq k \leq s$ and its reflection $(n-1)s \leq k \leq ns$, yielding~\eqref{eq:22} and~\eqref{eq:24}.

\begin{example}[Shorter sums through complementation]
\label{ex:shorter-sums}
For $(n,s)=(12,3)$, the support is $0\leq k\leq36$. Evaluating each of the $37$ coefficients separately by the direct support-sharp formula~\eqref{eq:BBK21} requires $190$ summands, whereas using the complementary total in~\eqref{eq:sharpsym} requires $105$:
\[
\sum_{k=0}^{36}\bigg(\bigg\lfloor\frac{k}{4}\bigg\rfloor+1\bigg) = 190,
\qquad
\sum_{k=0}^{36}\bigg(\bigg\lfloor\frac{\min(k,36-k)}{4}\bigg\rfloor+1\bigg) = 105.
\]
\end{example}

\subsection{The inhomogeneous boundary regimes}

Let $s_{\min} = \min(s_{1},\dots,s_{n})$. For arbitrary bounds, every solution in the lower range $0 \leq k \leq s_{\min}$ satisfies $k_{i} \leq k \leq s_{i}$, so that
\begin{equation}
\label{eq:27}
R_{\bm{s}}(k) = \C{k+n-1}{n-1}.
\end{equation}
In the upper range $N-s_{\min} \leq k \leq N$, the complement bijection gives
\begin{equation}
\label{eq:28}
R_{\bm{s}}(k) = \C{N-k+n-1}{n-1}.
\end{equation}
The complement map preserves the middle range $s_{\min} < k < N-s_{\min}$, but neither the bounds $k_{i} \leq s_{i}$ nor the complementary bounds $b_{i} \leq s_{i}$ are automatic there. By Proposition~\ref{prop:iff}, and by the same argument applied to the $k_{i}$ themselves, each one-term formula holds exactly on its own boundary range, so neither holds in the middle range. Equating \eqref{eq:27}--\eqref{eq:28} with \eqref{eq:9} produces inhomogeneous analogues of \eqref{eq:22} and \eqref{eq:24}. The alternating sums no longer share the factor $\C{n}{j}$, but the boundary values remain simplex, or figurate, numbers.


\section{A probabilistic interpretation}
\label{sec:prob}

Let $X_{1},\dots,X_{n}$ be independent and uniform on $\{0,1,\dots,s\}$, and write $U_{s}=\frac{1}{s+1}(\delta_{0}+\cdots+\delta_{s})$ for their common distribution. Set $S_{\!n}=X_{1}+\cdots+X_{n}$. Normalizing the coefficients in~\eqref{eq:bisdef} gives the probability mass function of $S_{\!n}$:
\begin{equation}
\label{eq:density}
\PP(S_{\!n}=k) = \frac{\displaystyle\bis{n}{k}{s}}{(s+1)^{n}}.
\end{equation}
Belbachir, Bouroubi, and Khelladi \cite[Corollary~5.1]{BelbachirBouroubi2008} give~\eqref{eq:density}, which recovers de Moivre's dice model in probabilistic form. The involution $X_{i} \mapsto s-X_{i}$ sends $S_{\!n}$ to $ns-S_{\!n}$, so the distribution is symmetric about $ns/2$, in agreement with~\eqref{eq:symmetry}.

\subsection{Moments}

Belbachir, Bouroubi, and Khelladi \cite[Corollary~5.2]{BelbachirBouroubi2008} evaluate the moments of $S_{\!n}$; the mean and variance follow directly from independence. A single draw from $U_{s}$ has mean $s/2$ and variance $\big((s+1)^{2}-1\big)/12$, so
\begin{equation}
\label{eq:moments}
\EE(S_{\!n}) = \frac{ns}{2},
\quad
\operatorname{Var}(S_{\!n}) = \frac{n{\mku}s(s+2)}{12} \eqd \tau_{n,s}^{2}.
\end{equation}
Symmetry about $ns/2$ makes every odd central moment vanish. For $s=1$, the variance becomes the binomial value $n/4$. Under the spin substitution $s=2\sigma$ of Section~\ref{sec:spin}, it becomes $\tau_{n,2\sigma}^{2} = n\sigma(\sigma+1)/3$, the familiar variance of the total magnetization of $n$ free spins.

\subsection{Mode and log-concavity}

For every $n \geq 1$ and $s \geq 0$, the coefficient sequence $k \mapsto \bis{n}{k}{s}$ is the $n$-fold convolution of $(1,1,\dots,1)$. This base sequence is log-concave with no internal zeros, and convolution preserves these properties \cite{BelbachirSzalay2008}. Hence the coefficient sequence is log-concave and unimodal. Together with symmetry, unimodality places a mode at the center. Belbachir \cite{Belbachir2011} gives the exact statement
\begin{equation}
\label{eq:mode}
\max_{0 \leq k \leq ns}\bis{n}{k}{s} = \bis{n}{\lfloor (ns+1)/2\rfloor}{s}.
\end{equation}
Deshouillers \cite{Deshouillers1974} attributes the location of this maximum at $\lfloor ns/2\rfloor$ to G.~L\'evy's 1971 thesis. For $n \geq 2$ and $s \geq 1$, the central mode is unique when $ns$ is even and there are exactly two central modes when $ns$ is odd. For $n=1$, all $s+1$ coefficients equal $1$; for $s=0$, the only coefficient is at $k=0$.

\subsection{Normal approximation and the central coefficients}
\label{sec:llt}

For fixed $s \geq 1$, the distribution $U_{s}$ has lattice span $1$, since its support contains consecutive integers, and moments of every order, since that support is bounded. Maximal span is precisely the condition for the uniform lattice local limit theorem \cite[Ch.~VII, Theorem~1]{Petrov1975}, and the third moment yields a version with an explicit rate \cite[Ch.~VII, Theorem~6]{Petrov1975}. With $\varphi$ denoting the standard normal density and $\tau_{n,s}$ defined in~\eqref{eq:moments},
\begin{equation}
\label{eq:llt}
\frac{\displaystyle\bis{n}{k}{s}}{(s+1)^{n}} = \frac{1}{\tau_{n,s}}\,\varphi\bigg(\frac{k-ns/2}{\tau_{n,s}}\bigg) + O(\tau_{n,s}^{-2}),
\end{equation}
uniformly in $k$ as $n\to\infty$.

For $k = \lfloor ns/2\rfloor$, the distance $|k-ns/2|$ is at most $1/2$, so the argument of $\varphi$ in~\eqref{eq:llt} is at most $1/2\tau_{n,s}$, and $\varphi$ may be replaced there by $\varphi(0) = 1/\sqrt{2\pi}$ at the cost of a relative error $O(\tau_{n,s}^{-2})$. The normalized central coefficient is therefore $\varphi(0)/\tau_{n,s}+O(\tau_{n,s}^{-2})$, with relative error $O(n^{-1/2})$, and substituting $\tau_{n,s}$ from~\eqref{eq:moments} gives
\begin{equation}
\label{eq:centralasym}
\bis{n}{\lfloor ns/2\rfloor}{s}
 \,\sim\,  (s+1)^{n}\sqrt{\frac{6}{\pi{\mku}n{\mku}s(s+2)}}.
\end{equation}
The central coefficient therefore accounts for a $\Theta(n^{-1/2})$ fraction of the total mass $(s+1)^{n}$. Since $\tau_{n,s} = \Theta(\sqrt{n})$ while the support has width $ns = \Theta(n)$, the mass concentrates on a vanishing fraction of the support. Equation~\eqref{eq:centralasym} also estimates the cardinality of the central rank layer used in Section~\ref{sec:hash}.

With $p=s+1$, so that $\bis{n}{k}{p-1}$ is the $p$-nomial coefficient of order $n$, equation~\eqref{eq:centralasym} becomes
\begin{equation}
\label{eq:pnomialmax}
\max_{k}\bis{n}{k}{p-1} \sim p^{n}\!\sqrt{\frac{6}{\pi{\mku}n\,(p^{2}-1)}}.
\end{equation}
The asymptotic formula~\eqref{eq:pnomialmax} was conjectured by G.~L\'evy and proved by Deshouillers \cite{Deshouillers1974}. Eger \cite{Eger2014} later recovered the same leading term by a lattice local-limit argument under the name \emph{central extended binomial coefficients}.

After division by $(s+1)^{n}$, the adjacent central difference is of order $\tau_{n,s}^{-3}$, whereas the uniform error in~\eqref{eq:llt} is only $O(\tau_{n,s}^{-2})$. Subtracting the approximation at two adjacent lattice points therefore does not control the difference. Petrov \cite[Ch.~VII, Theorem~13]{Petrov1975} refines the lattice local limit theorem to an asymptotic expansion with a uniform remainder of arbitrary order in $n^{-1/2}$, since the bounded support of $U_{s}$ guarantees the required moments. Neuschel \cite{Neuschel2014} gives a complete local expansion with uniform error terms; alternatively, one may use the midpoint-anchored saddle-point expansion of \cite{Mendonca2026}. Petrov's expansion through order $n^{-1}$ already suffices. With $x = (k-ns/2)/\tau_{n,s}$, and with $\kappa_{4} = -\big((s+1)^{4}-1\big)/120$ the fourth cumulant of $U_{s}$, it reads
\begin{equation}
\label{eq:edgeworth}
\tau_{n,s}\frac{\bis{n}{k}{s}}{(s+1)^{n}} = \varphi(x)\bigg[1+\frac{n\kappa_{4}}{24\tau_{n,s}^{4}}\big(x^{4}-6x^{2}+3\big)\bigg] + o(n^{-1}),
\end{equation}
uniformly in $k$. The term of order $n^{-1/2}$ is proportional to the third cumulant of $U_{s}$ and vanishes by symmetry. Had it survived, it would contribute to the adjacent difference at the same order as the Gaussian term. The correction that remains, $\varphi(x)$ times the $O(n^{-1})$ polynomial in the bracket, is even in $x$, so between $x=0$ and $x=1/\tau_{n,s}$, that is, between $k=ns/2$ and $k=ns/2+1$, it changes by only $O(n^{-1}\tau_{n,s}^{-2})$. The Gaussian term changes by
\begin{equation}
\varphi(0)-\varphi(1/\tau_{n,s}) = \varphi(0)\big(1-e^{-1/2\tau_{n,s}^{2}}\big) = \frac{\varphi(0)}{2\tau_{n,s}^{2}}+O(\tau_{n,s}^{-4}).
\end{equation}
Since $\tau_{n,s}^{2}$ is of order $n$, every other contribution is $o(\tau_{n,s}^{-2})$, and dividing by $\tau_{n,s}$ gives, for fixed $s$ and even $ns$,
\begin{equation}
\label{eq:centraldiffasym}
\bis{n}{ns/2}{s}-\bis{n}{ns/2+1}{s} \,\sim\, 
\frac{(s+1)^{n}}{2\!\sqrt{2\pi}\,\tau_{n,s}^{3}}, \quad n \to \infty.
\end{equation}
Thus the normalized central difference is $\Theta(n^{-3/2})$, one factor of $n^{-1}$ below the central probability mass, which is $\Theta(n^{-1/2})$. This extra factor is characteristic of the Catalan-type first differences considered below.

Under the spin dictionary of Section~\ref{sec:spin}, let $d^{\,(J)}_{n,\sigma}$ denote the multiplicity of the spin-$J$ irreducible representation in the $n$-fold tensor product of spin-$\sigma$ representations, as defined in~\eqref{eq:dJ}. In the central regime $J = O(\sqrt{n})$, these multiplicities have a Maxwell-type profile, a Gaussian in $J/\sqrt{n}$ multiplied by $2J+1$. Polychronakos and Sfetsos \cite{PolychronakosSfetsos2016} obtained this profile through random walks, while Curtright, Van Kortryk, and Zachos \cite{CurtrightVanKortryk2017} derived it from $\mathrm{SU}(2)$ characters.

The midpoint-anchored saddle-point expansion of \cite{Mendonca2026} recovers the central profile and extends it to a large-deviation asymptotic with relative error $O(n^{-1})$, uniformly for fixed $\sigma$ and any $\epsilon > 0$ throughout $0 \leq (J+\tfrac{1}{2})/n \leq \sigma-\epsilon$. Its central specialization recovers the leading fixed-$J$ form of Curtright, Van Kortryk, and Zachos \cite{CurtrightVanKortryk2017} and the large-$J$ central form of Polychronakos and Sfetsos \cite{PolychronakosSfetsos2016}. When $n\sigma \in \ZZ$, equation~\eqref{eq:centraldiffasym} with $s=2\sigma$ is its $J=0$ specialization and yields the singlet asymptotic in the following example.

\begin{example}
\label{ex:singlet}
Take $n=2\nu$ spin-$2$ particles. Then $s=4$, $\tau_{2\nu,4}^{2} = 2\nu \cdot 4 \cdot 6/12 = 4\nu$, and $2\sqrt{2\pi}\,\tau_{2\nu,4}^{3} = 16\sqrt{2\pi}\,\nu^{3/2}$, so equation~\eqref{eq:centraldiffasym} gives
\begin{equation}
d^{\,(0)}_{2\nu,2} \,\sim\, \frac{5^{2\nu}}{16\sqrt{2\pi}\,\nu^{3/2}}.
\end{equation}
This recovers the leading behavior of sequence \href{https://oeis.org/A007043}{A007043} in Table~\ref{tab:oeis}. Taking $\sigma=\tfrac{1}{2}$ instead, $\tau_{2\nu,1}^{2} = \nu/2$, and the same computation gives the classical Catalan asymptotic $4^{\nu}/\sqrt{\pi}\,\nu^{3/2}$.
\end{example}

\begin{corollary}[Central asymptotic corrections]
\label{cor:central-corrections}
For fixed $s \geq 1$ and even $ns$, as $n \to \infty$,
\begin{align}
\label{eq:centralcoeff-correction}
\bis{n}{ns/2}{s}
&= \frac{(s+1)^{n}}{\sqrt{2\pi}\,\tau_{n,s}}
\bigg[1-\frac{3((s+1)^{2}+1)}{20((s+1)^{2}-1)n}+O(n^{-2})\bigg], \\[1ex]
\label{eq:centraldiff-correction}
\bis{n}{ns/2}{s}-\bis{n}{ns/2+1}{s}
&= \frac{(s+1)^{n}}{2\sqrt{2\pi}\,\tau_{n,s}^{3}}
\bigg[1-\frac{3((s+1)^{2}+5)}{4((s+1)^{2}-1)n}+O(n^{-2})\bigg].
\end{align}
\end{corollary}

\begin{proof}
Apply Neuschel's uniform local expansion \cite{Neuschel2014} through order $n^{-3}$. Its correction terms are smooth and even, so their differences between $x=0$ and $x=1/\tau_{n,s}$ gain a factor $O(n^{-1})$; the uniform remainder is $o(n^{-3})$. The leading difference of the rescaled probabilities in~\eqref{eq:edgeworth} is of order $n^{-1}$, so passing to a relative remainder amplifies the uniform remainder by a factor of order $n$. With $\kappa_{2}=s(s+2)/12$ and $\lambda_{4}=\kappa_{4}/\kappa_{2}^{2}=-6((s+1)^{2}+1)/5((s+1)^{2}-1)$, the relative corrections are $\lambda_{4}/8n$ and $5\lambda_{4}/8n-1/4\tau_{n,s}^{2}$, respectively.
\end{proof}

Under $s=2\sigma$, equation~\eqref{eq:centraldiff-correction} recovers the general singlet correction of Curtright, Van Kortryk, and Zachos \cite[eq.~(33)]{CurtrightVanKortryk2017}; for $s=2$, its $1/n$ coefficient is $-21/16$.

\begin{remark}[Higher-order corrections]
\label{rem:central-higher-orders}
Carrying the same expansion through order $n^{-4}$ gives the next terms in the brackets of \eqref{eq:centralcoeff-correction} and \eqref{eq:centraldiff-correction}, respectively,
\begin{equation}
\label{eq:central-second-corrections}
\frac{-13(s+1)^{4}+134(s+1)^{2}-13}{1120((s+1)^{2}-1)^{2}{\mku}n^{2}},
\qquad
\frac{29(s+1)^{4}+1058(s+1)^{2}+1829}{160((s+1)^{2}-1)^{2}{\mku}n^{2}}.
\end{equation}
With these terms retained, both remainders improve to $O(n^{-3})$.
\end{remark}

\begin{remark}[Central coefficients and nonnegative walks]
\label{rem:survival}
For fixed $s \geq 1$, the centered sums $T_{\!j} = S_{\!j}-js/2$ form a symmetric random walk. Put $u_{0,s}=1$ and $u_{n,s} = \PP(T_{1} \geq 0, \dots, T_{n} \geq 0)$. Its return probabilities are
\[
p_{m,s}=\PP(T_{m}=0) = 
\begin{cases}
\dfrac{\displaystyle\bis{m}{ms/2}{s}}{(s+1)^{m}}, & ms \text{ even}, \\[6pt]
\hfill 0 \hfill, & ms \text{ odd}.
\end{cases}
\]
The generalized Sparre Andersen identity, valid also for discrete jumps \cite[\S III.2]{MounaixMajumdarSchehr2020}, gives
\begin{equation}
\label{eq:survival-gf}
\sum_{n \geq 0}u_{n,s}{\mku}z^{n}
= \exp\bigg(\sum_{m \geq 1}\frac{z^{m}}{m}\PP(T_{m} \geq 0)\bigg)
= \frac{1}{\sqrt{1-z}}\exp\bigg(\frac{1}{2}\sum_{m \geq 1}\frac{p_{m,s}{\mku}z^{m}}{m}\bigg), \quad |z|<1,
\end{equation}
where symmetry gives $\PP(T_{m} \geq 0)=(1+p_{m,s})/2$. Thus the central coefficients determine the enumeration $(s+1)^{n}u_{n,s}$ of words $(x_{1},\dots,x_{n}) \in \{0,\dots,s\}^{n}$ for which every prefix satisfies $x_{1}+\cdots+x_{j} \geq js/2$. For $s=1$, the reflection principle gives $2^{n}u_{n,1} = \C{n}{\lfloor n/2\rfloor}$. The return term in \eqref{eq:survival-gf} retains the distinction between weak and strict positivity on the lattice.
\end{remark}


\section{Products of chains and first differences}
\label{sec:chains}

\subsection{Rank layers and symmetric chains}

Consider the product
\begin{equation}
P_{\bm{s}} = \prod_{i=1}^{n}\{0,1,\dots,s_{i}\}
\end{equation}
with the componentwise order and the rank function $\rho(k_{1},\dots,k_{n}) = k_{1}+\cdots+k_{n}$. Its rank-$k$ elements are precisely the bounded compositions in~\eqref{eq:problem}, so the rank-$k$ layer has cardinality $R_{\bm{s}}(k)$. In the homogeneous case, the rank generating polynomial is $(1+x+\cdots+x^{s})^{n}$.

A saturated chain in $P_{\bm{s}}$ is \emph{symmetric} if the ranks of its first and last elements sum to $N$. De Bruijn, van Ebbenhorst Tengbergen, and Kruyswijk \cite{deBruijnTengbergen1951} proved that every finite product of chains admits a partition into symmetric chains, or a symmetric chain decomposition. Stanley \cite{Stanley1980} established the connection between the hard Lefschetz theorem and Sperner properties. Proctor \cite{Proctor1982} gave a linear realization through an $\mathfrak{sl}(2,\CC)$ action whose raising operator has maximal rank between consecutive layers; this yields unimodality and the Sperner property. See Srinivasan \cite{Srinivasan2011} for a modern account.

The de Bruijn--Tengbergen--Kruyswijk construction is inductive. Given a symmetric chain decomposition of $P_{(s_{1},\dots,s_{n-1})}$, let $C = \{c_{0} < c_{1} < \cdots < c_{a-1}\}$ be one of its chains and put $b=s_{n}+1$. Symmetry of $C$ means that $\rho(c_{0})+\rho(c_{a-1}) = N-s_{n}$. The product $C \times \{0,\dots,s_{n}\}$ is an $a \times b$ grid, and its outer hook
\begin{equation}
(c_{0},0) < (c_{1},0) < \cdots < (c_{a-1},0) < (c_{a-1},1) < \cdots < (c_{a-1},s_{n})
\end{equation}
is a saturated chain with $a+b-1$ elements. Its end ranks $\rho(c_{0})$ and $\rho(c_{a-1})+s_{n}$ sum to $N$, so it is symmetric. Removing it leaves the $(a-1) \times (b-1)$ grid $\{c_{0},\dots,c_{a-2}\} \times \{1,\dots,s_{n}\}$, whose end ranks $\rho(c_{0})+1$ and $\rho(c_{a-1})+s_{n}-1$ again sum to $N$. Iteration produces $\min\{a,b\}$ symmetric chains of lengths
\begin{equation}
\label{eq:gridsplit}
a+b-1, \quad a+b-3, \quad \dots, \quad |a-b|+1.
\end{equation}
Figure~\ref{fig:grid-cg} displays the $5 \times 4$ case beside the Clebsch--Gordan decomposition with the same dimension profile.

\begin{figure}[!t]
\centering
\begin{tikzpicture}[x=7mm,y=7mm,line cap=round,line join=round]
\foreach \y in {0,...,3}{
  \draw[black!20] (0,\y) -- (4,\y);
}
\foreach \x in {0,...,4}{
  \draw[black!20] (\x,0) -- (\x,3);
}

\draw[black,line width=1.1pt] (0,0) -- (4,0) -- (4,3);
\draw[black!75,line width=1.1pt,densely dashed] (0,1) -- (3,1) -- (3,3);
\draw[black!55,line width=1.1pt,dashdotted] (0,2) -- (2,2) -- (2,3);
\draw[black!40,line width=1.1pt,densely dotted] (0,3) -- (1,3);

\foreach \x/\y in {0/0,1/0,2/0,3/0,4/0,4/1,4/2,4/3}{
  \filldraw[fill=black,draw=black] (\x,\y) circle (1.45pt);
}
\foreach \x/\y in {0/1,1/1,2/1,3/1,3/2,3/3}{
  \filldraw[fill=black!75,draw=black] (\x,\y) circle (1.45pt);
}
\foreach \x/\y in {0/2,1/2,2/2,2/3}{
  \filldraw[fill=black!55,draw=black] (\x,\y) circle (1.45pt);
}
\foreach \x/\y in {0/3,1/3}{
  \filldraw[fill=black!40,draw=black] (\x,\y) circle (1.45pt);
}

\node[font=\small,align=center] at (2,-0.95) {nested hooks \\ lengths $8,6,4,2$};
\draw[line width=0.75pt,<->] (4.7,1.5) -- (6.0,1.5);
\node[font=\small,align=center] at (9.,1.5) {$V_{2}\otimes V_{3/2}$ \\[3pt] $\cong V_{7/2} \oplus V_{5/2} \oplus V_{3/2} \oplus V_{1/2}$ \\[3pt] dimensions $8,6,4,2$};
\end{tikzpicture}
\caption{Nested-hook decomposition of a $5 \times 4$ product grid. The four symmetric saturated chains have lengths $8,6,4,2$. Since $5=2(2)+1$ and $4=2(3/2)+1$, these four chain lengths are the dimensions of the Clebsch--Gordan constituents displayed on the right. The four chain lengths and the four constituent dimensions coincide as multisets.}
\label{fig:grid-cg}
\end{figure}
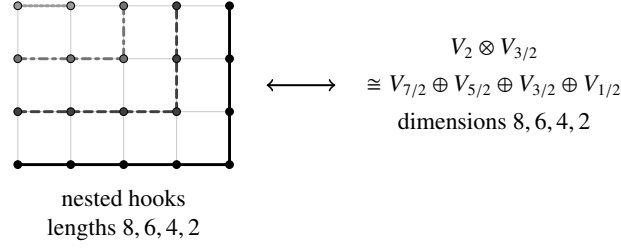

Starting from the decomposition of the empty product into a single one-element chain and applying this nested-hook split to every chain at each step $i=1,\dots,n$ yields a symmetric chain decomposition of $P_{\bm{s}}$.

If $k < N/2$, every symmetric chain that meets rank $k$ continues through rank $k+1$, because it runs from some rank $c \leq k$ to the rank $N-c \geq N-k > k$. Sending each rank-$k$ element to its successor on the chain therefore injects the rank-$k$ layer into the rank-$(k+1)$ layer and proves unimodality. Reversing every chain gives a bijection between ranks $k$ and $N-k$, recovering symmetry. Moreover, every symmetric chain meets rank $\lfloor N/2\rfloor$ exactly once, so the number of chains equals $R_{\bm{s}}(\lfloor N/2\rfloor)$. Since an antichain meets each chain in at most one element, a central rank layer is a largest antichain. This is the Sperner property used in the coding application of Section~\ref{sec:hash}.

The same decomposition proves the strong Sperner property. Let $1 \leq r \leq N+1$, and let $A\subseteq P_{\bm{s}}$ contain no chain of $r+1$ elements. For every symmetric chain $C$, we have $|A \cap C| \leq \min\{r,|C|\}$. Because the ranks occupied by $C$ form an interval centered at $N/2$, the $r$ rank layers nearest the center contain exactly $\min\{r,|C|\}$ elements of $C$. When $r$ and $N$ have the same parity, these $r$ layers can be chosen in two ways, and either choice works, because every symmetric chain has $|C| \equiv N+1$~(mod $2$) and hence $|C| \neq r$. Summing over the decomposition shows that $|A|$ is at most the sum of the $r$ largest rank sizes, and the union of those $r$ central layers attains the bound.

Any symmetric chain decomposition of $P_{\bm{s}}$ has exactly $R_{\bm{s}}(\lfloor N/2\rfloor)$ chains, which in the homogeneous case is asymptotically $(s+1)^{n}\sqrt{6/\pi n s(s+2)}$ by~\eqref{eq:centralasym}. The homogeneous poset has $(s+1)^{n}$ elements. Hence, for fixed $s \geq 1$, the average chain length in any symmetric chain decomposition is
\begin{equation}
\frac{(s+1)^{n}}{\displaystyle\bis{n}{\lfloor ns/2\rfloor}{s}}
\sim \sqrt{\frac{\pi n s(s+2)}{6}}.
\end{equation}
The average length is therefore $\Theta(\!\sqrt{n})$, whereas the longest chain has $ns+1 = \Theta(n)$ elements.

\subsection{First differences and chain starts}

Fix a symmetric chain decomposition of $P_{\bm{s}}$ and set $R_{\bm{s}}(-1)=0$. Each chain contains at most one element of any given rank, so $R_{\bm{s}}(k)$ is the number of chains meeting rank $k$. For $0 \leq k \leq \lfloor N/2\rfloor$, every chain meeting rank $k-1$ also meets rank $k$. Hence the number of chains beginning at rank $k$ is
\begin{equation}
\label{eq:chainstarts}
R_{\bm{s}}(k)-R_{\bm{s}}(k-1).
\end{equation}
In the homogeneous case, define the first differences by
\begin{equation}
\label{eq:Ddef}
D_{n,s}(k) = \bis{n}{k}{s}-\bis{n}{k-1}{s}.
\end{equation}
Multiplying the generating polynomial by $1-x$ gives
\begin{equation}
\label{eq:Dgf}
\sum_{k \geq 0}D_{n,s}(k)x^{k} = 
(1-x)(1+x+\cdots+x^{s})^{n} = \frac{(1-x^{s+1})^{n}}{(1-x)^{n-1}}.
\end{equation}
For $n \geq 2$, expand $(1-x^{s+1})^{n}$ and use $[x^{m}](1-x)^{-(n-1)} = \C{m+n-2}{n-2}$ for $m \geq 0$. Coefficient extraction gives the exact first-difference formula
\begin{equation}
\label{eq:Dsum}
D_{n,s}(k) = \sum_{j=0}^{\min\{n,\lfloor k/(s+1)\rfloor\}}
(-1)^{j}{\mku}\C{n}{j}{\mku}\C{k-j(s+1)+n-2}{n-2},
\quad k \geq 0.
\end{equation}
In particular, throughout the lower boundary range $0 \leq k \leq s$, only the $j=0$ term contributes and
\begin{equation}
\label{eq:Dboundary}
D_{n,s}(k) = \C{k+n-2}{n-2}.
\end{equation}

For $0 \leq k \leq \lfloor ns/2\rfloor$, $D_{n,s}(k)$ counts the chains beginning at rank $k$. Such a chain ends at rank $ns-k$ and therefore contains $ns-2k+1$ elements. With $s=2\sigma$, this is $2(n\sigma-k)+1$, the dimension of a spin-$(n\sigma-k)$ irreducible representation.

\subsection{Symmetric chains and Clebsch--Gordan profiles}
\label{sec:cg}

Fix $s_{i}=2\sigma_{i}$, with $\sigma_{i} \in \tfrac{1}{2}\ZZ_{\geq 0}$, and let $V_{\sigma}$ denote the irreducible $\mathrm{SU}(2)$-module of spin $\sigma$, of dimension $2\sigma+1$. Assigning the coordinate value $k_{i}$ the weight $k_{i}-\sigma_{i}$ identifies the rank-$k$ elements of $P_{\bm{s}}$ with a basis of the weight space of weight $k-N/2$ in $\bigotimes_{i=1}^{n}V_{\sigma_{i}}$.

Proctor's $\mathfrak{sl}(2,\CC)$ action on the vector space with basis $P_{\bm{s}}$ \cite{Proctor1982} realizes the module $\bigotimes_{i}V_{\sigma_{i}}$; its raising operator is the sum of the $n$ single-coordinate raising operators. An irreducible decomposition of this module and a symmetric chain decomposition of the underlying poset are different objects: the former expresses a vector space as a direct sum of invariant subspaces, whereas the latter partitions the elements of a poset. Srinivasan \cite{Srinivasan2011} constructs a symmetric Jordan basis that supplies a linear bridge between them.

\begin{proposition}[Chain and constituent profiles]
\label{prop:scd-isotypic}
Let $\mathcal{C}$ be any symmetric chain decomposition of $P_{\bm{s}}$ with $s_{i}=2\sigma_{i}$. For $0 \leq k \leq \lfloor N/2\rfloor$, the number of chains in $\mathcal{C}$ beginning at rank $k$ is
\begin{equation}
\label{eq:scd-mult}
R_{\bm{s}}(k)-R_{\bm{s}}(k-1) \, = \, d^{\,(N/2-k)}_{\bm{\sigma}},
\end{equation}
the multiplicity of $V_{N/2-k}$ in $\bigotimes_{i}V_{\sigma_{i}}$. Consequently, the multiset of chain lengths in $\mathcal{C}$ equals the multiset of dimensions of the irreducible constituents. After choosing an irreducible direct-sum decomposition, we can match chains and constituent copies noncanonically, pairing each chain beginning at rank $k$ with a copy of $V_{N/2-k}$.
\end{proposition}

\begin{proof}
By \eqref{eq:chainstarts}, the number of chains beginning at rank $k$ is $R_{\bm{s}}(k)-R_{\bm{s}}(k-1)$. Let $w_{m}$ denote the dimension of the weight-$m$ space and put $J=N/2-k$. A spin-$L$ summand contributes one vector at every weight from $-L$ to $L$. Hence $w_{-J}$ and $w_{-(J+1)}$ are the total multiplicities of constituents with spin at least $J$ and $J+1$, respectively. Their difference isolates the spin-$J$ multiplicity:
\begin{equation}
d^{\,(J)}_{\bm{\sigma}} = w_{-J}-w_{-(J+1)}.
\end{equation}
The weight-space dictionary gives $w_{-J} = R_{\bm{s}}(k)$ and $w_{-(J+1)} = R_{\bm{s}}(k-1)$, proving \eqref{eq:scd-mult}. A chain beginning at rank $k$ ends at rank $N-k$ and therefore has $N-2k+1$ elements, the dimension of $V_{N/2-k}$. Pair chains and constituent copies separately for each starting rank.
\end{proof}

The Clebsch--Gordan decomposition of the tensor product $V_{\alpha}\otimes V_{\beta}$ is
\begin{equation}
\label{eq:cg}
V_{\alpha}\otimes V_{\beta} \,\cong\,
\bigoplus_{J=|\alpha-\beta|}^{\alpha+\beta} V_{J},
\end{equation}
where $J$ increases in unit steps. This tensor-product decomposition is the representation-theoretic counterpart of the grid rule~\eqref{eq:gridsplit}. With $a=2\alpha+1$ and $b=2\beta+1$, its constituents have dimensions $a+b-1,a+b-3,\dots,|a-b|+1$, and there are $\min\{a,b\}$ of them.

\begin{remark}
\label{rem:cg}
The de Bruijn--Tengbergen--Kruyswijk nested-hook induction therefore mirrors the Clebsch--Gordan rule at the level of profiles. The conclusion is an equality of length--dimension profiles, not a canonical identification of individual chains with irreducible summands.
\end{remark}

For $0 \leq k \leq \lfloor ns/2\rfloor$, equation~\eqref{eq:Dsum} is a character-free closed form for the multiplicity $d^{\,(ns/2-k)}_{n,s/2}$.

\begin{example}[A binary row and its interpretations]
\label{ex:binary-dictionary}
For $(n,s)=(4,1)$, Table~\ref{tab:bisnomial-triangles} gives the coefficient row $(1,4,6,4,1)$, whose lower-half first differences are $(1,3,2)$. The middle layer contains the six binary words with exactly two ones, so four independent fair binary draws total two with probability $6/16=3/8$. Every symmetric chain decomposition of $\{0,1\}^{4}$ has one chain of length five, three of length three, and two of length one. Proposition~\ref{prop:scd-isotypic} gives the matching constituent profile
\[
V_{1/2}^{\,\otimes 4} \cong V_{2} \oplus 3V_{1} \oplus 2V_{0},
\quad 5+3 \cdot 3+2 \cdot 1 = 16.
\]
Thus the zero-magnetization subspace has dimension six, whereas the singlet subspace has dimension $6-4=2$: each spin-$2$ or spin-$1$ constituent also contributes one zero-weight vector.
\end{example}

Proposition~\ref{prop:scd-isotypic} and the profile comparison hold for arbitrary bounds, hence for mixed spins. We work out the mixed-spin case in Section~\ref{sec:extensions}.


\section{Representation theory and applications}
\label{sec:conn}

\subsection{Spin multiplicities}
\label{sec:spin}

Specialize the dictionary of Section~\ref{sec:cg} to $n$ identical spin-$\sigma$ particles, $\sigma \in \tfrac{1}{2}\ZZ_{\geq 0}$. If $m$ denotes the total magnetization, use the identifications
\begin{equation}
\label{eq:spindict}
s = 2\sigma, \quad k = m+\sigma n.
\end{equation}
Here $m+n\sigma \in \ZZ$, and the spin labels satisfy $0 \leq J \leq n\sigma$ and $J-n\sigma \in \ZZ$. The coefficient $\bis{n}{m+\sigma n}{\,2\sigma}$ is the dimension of the total-magnetization-$m$ subspace. The multiplicity $d^{\,(J)}_{n,\sigma}$ of the spin-$J$ irreducible representation in the $n$-fold tensor product of spin-$\sigma$ representations is the difference between the dimensions of the magnetization-$J$ and magnetization-$(J+1)$ subspaces:
\begin{equation}
\label{eq:dJ}
d^{\,(J)}_{n,\sigma} = \bis{n}{J+\sigma n}{2\sigma}-\bis{n}{J+\sigma n+1}{2\sigma}.
\end{equation}

The midpoint shift also has a representation-theoretic interpretation. In spin units, the Weyl vector of $\mathfrak{su}(2)$ is $\rho=\tfrac{1}{2}$, and
\[
\dim V_{J}=2(J+\rho), \qquad J(J+1)=(J+\rho)^{2}-\rho^{2}.
\]
Thus the Gaussian factor in the central multiplicity profile of Section~\ref{sec:llt} is expressed in terms of the shifted spin $J+\tfrac{1}{2}$, or equivalently through the shifted Casimir $J(J+1)+\tfrac{1}{4}$.

To measure the deficit from maximal spin, set $\ell = n\sigma-J$. Then $J+\sigma n = 2\sigma n-\ell$, and the symmetry relation~\eqref{eq:symmetry}, with $ns = 2\sigma n$, replaces the lower entries $J+\sigma n$ and $J+\sigma n+1$ in~\eqref{eq:dJ} by $\ell$ and $\ell-1$:
\begin{equation}
\label{eq:ddeficit}
d^{\,(n\sigma-\ell)}_{n,\sigma} = \bis{n}{\ell}{2\sigma}-\bis{n}{\ell-1}{2\sigma} = D_{n,2\sigma}(\ell).
\end{equation}
Proposition~\ref{prop:scd-isotypic} gives a second interpretation of \eqref{eq:ddeficit}: $D_{n,2\sigma}(\ell)$ is the number of symmetric chains beginning at rank $\ell$ in $\{0,1,\dots,2\sigma\}^{n}$, and each such chain has $2J+1$ elements.

The boundary range $0 \leq k \leq s$ corresponds to $-\sigma n \leq m \leq -\sigma(n-2)$. Its reflection under $m\mapsto -m$ gives the opposite boundary range, so their union is the near-saturation sector
\begin{equation}
\label{eq:saturation}
|m| \geq \sigma(n-2).
\end{equation}
In this sector, the fixed-magnetization subspaces have figurate-number dimensions, and the irreducible multiplicities are their first differences. For $n \geq 2$ and $0 \leq \ell \leq 2\sigma$, equations~\eqref{eq:ddeficit} and \eqref{eq:Dboundary} give
\begin{equation}
\label{eq:nearspin}
d^{\,(n\sigma-\ell)}_{n,\sigma} = \C{\ell+n-2}{n-2}.
\end{equation}
For $\sigma \geq 1$, the values at $\ell=0,1,2$ are $1$, $n-1$, and $\C{n}{2}$; when $\sigma=\tfrac{1}{2}$, only $\ell=0,1$ lie in the boundary range. The distinction between weight-space dimension and irreducible multiplicity is already visible at $\ell=1$: the magnetization-$(n\sigma-1)$ subspace has dimension $n$, whereas spin $J=n\sigma-1$ occurs with multiplicity $n-1$.

When $n=2\nu$ is even, the central first difference
\begin{equation}
\label{eq:29}
d^{\,(0)}_{2\nu,\sigma} = \bis{2\nu}{2\sigma\nu}{2\sigma} - \bis{2\nu}{2\sigma\nu+1}{2\sigma}
\end{equation}
is the dimension of the $\mathrm{SU}(2)$-invariant subspace of the $2\nu$-fold tensor product of the $(2\sigma+1)$-dimensional irreducible representation.

Equation~\eqref{eq:29} recovers familiar low-spin sequences. For $\sigma=\tfrac{1}{2}$ it gives the Catalan numbers $\C{2\nu}{\nu}-\C{2\nu}{\nu+1}$, and for $\sigma=1$ the even-index Riordan numbers. For $\sigma=2$ it gives the dimensions of the spaces of even-degree noncommutative $\mathrm{SL}(2,\CC)$-invariants in five variables. Almkvist, Dicks, and Formanek \cite{AlmkvistDicks1985} studied these invariant spaces before the spin-multiplicity interpretation. These sequences are catalogued in the OEIS; Table~\ref{tab:oeis} gives the concordance.

\begin{table}[htbp]
\centering
\caption{The singlet sequences \eqref{eq:29} and their OEIS entries. For
$\sigma=1$ and $\sigma=2$, the table lists the even-index subsequence of the
cited entry.}
\label{tab:oeis}
\medskip
\begin{tabular}{@{}cccc@{}}
\toprule
$\sigma$ & $s=2\sigma$ & $d^{\,(0)}_{2\nu,\sigma}$,\, $\nu \geq 0$ & OEIS concordance \\
\midrule
$\frac{1}{2}$ & $1$ & $1,\, 1,\, 2,\, 5,\, 14,\, 42,\, 132,\, \dots$ & \href{https://oeis.org/A000108}{A000108}$(\nu)$ (Catalan) \\[3pt]
$1$ & $2$ & $1,\, 1,\, 3,\, 15,\, 91,\, 603,\, 4213,\, \dots$ & \href{https://oeis.org/A005043}{A005043}$(2\nu)$ (Riordan) $=$ \href{https://oeis.org/A099251}{A099251}$(\nu)$ \\[3pt]
$\frac{3}{2}$ & $3$ & $1,\, 1,\, 4,\, 34,\, 364,\, 4269,\, 52844,\, \dots$ & \href{https://oeis.org/A264607}{A264607}$(\nu)$ \\[3pt]
$2$ & $4$ & $1,\, 1,\, 5,\, 65,\, 1085,\, 19845,\, 383251,\, \dots$ & \href{https://oeis.org/A007043}{A007043}$(2\nu)$ \\[3pt]
$\frac{5}{2}$ & $5$ & $1,\, 1,\, 6,\, 111,\, 2666,\, 70146,\, 1949156,\, \dots$ & \href{https://oeis.org/A272391}{A272391}$(\nu)$ \\
\bottomrule
\end{tabular}
\end{table}

The two Riordan entries record complementary interpretations. OEIS entry \href{https://oeis.org/A099251}{A099251} identifies the even-index Riordan numbers with the Haar moments $\langle(\operatorname{tr}O)^{2\nu}\rangle$ over the $3 \times 3$ orthogonal group. The double cover $\mathrm{SU}(2) \to \mathrm{SO}(3)$ mediates this correspondence. OEIS entry \href{https://oeis.org/A005043}{A005043} records the Motzkin- and Riordan-path descriptions.

\begin{remark}[Sequential spin coupling]
\label{rem:spin-chain}
Iterating the Clebsch--Gordan rule \eqref{eq:cg} identifies $d^{\,(J)}_{n,\sigma}$ with the number of coupling paths $J_{0} = 0, J_{1}, \dots, J_{n} = J$, where $J_{j+1}$ runs from $|J_{j}-\sigma|$ to $J_{j}+\sigma$ in unit steps. The dimension identity $\sum_{L=|J-\sigma|}^{J+\sigma}(2L+1) = (2J+1)(2\sigma+1)$ makes
\begin{equation}
\label{eq:spin-chain-kernel}
K_{\sigma}(J,L) = \frac{2L+1}{(2J+1)(2\sigma+1)},
\quad 
L = |J-\sigma|,|J-\sigma|+1,\dots,J+\sigma,
\end{equation}
with zero entries otherwise, a Markov transition kernel on the spin labels. This is the $\mathrm{SU}(2)$ specialization of the tensor-product chain of Biane, Bougerol, and O'Connell \cite[\S 5.5]{BianeBougerolOConnell2005}. Along a path, the dimension factors telescope, so starting at $J_{0}=0$ gives
\begin{equation}
\label{eq:spin-chain-law}
\PP(J_{n}=J) = \frac{(2J+1)d^{\,(J)}_{n,\sigma}}{(2\sigma+1)^{n}}.
\end{equation}
This is the fraction of the full tensor-product dimension belonging to the spin-$J$ sector. For $\sigma=\tfrac12$, the coupling paths ending at zero are Dyck paths after doubling their heights. For $\sigma=1$, they are Riordan paths: non-negative paths with steps $-1,0,1$ and no horizontal step at height zero, since $V_{0} \otimes V_{1} \cong V_{1}$. This boundary rule distinguishes them from ordinary non-negative bridges of the centered walk in Remark~\ref{rem:survival}.
\end{remark}

\subsection{\texorpdfstring{The $s$-Catalan numbers as Littlewood--Richardson coefficients}{The s-Catalan numbers as Littlewood-Richardson coefficients}}
\label{sec:scat}

Belbachir and Igueroufa \cite[Def.~3.2]{BelbachirIgueroufa2020a} name the central first differences \eqref{eq:29}
\begin{equation}
\label{eq:scat}
\mathrm{Cat}^{(s)}_{\nu} \deq \bis{2\nu}{s\nu}{s}-\bis{2\nu}{s\nu+1}{s}
\end{equation}
the \emph{$s$-Catalan numbers}, under a standing hypothesis that $s$ is odd, and already observe that these values occur as spin multiplicities \cite{CohenHansen2016,CurtrightVanKortryk2017}. Linz \cite{Linz2022} keeps the name for every $s \geq 1$, uses the notation $\bis{n}{k}{s}$ adopted here alongside the term \emph{$s$-binomial}, and tabulates the values in his Figure~1. For $s=2\sigma$, $\mathrm{Cat}^{(s)}_{\nu}=d^{\,(0)}_{2\nu,\sigma}$.

For a partition $\lambda=(\lambda_{1},\dots,\lambda_{r})$, let $s*\lambda = (s\lambda_{1},\dots,s\lambda_{r})$ denote its horizontal stretch by a factor $s$. Figure~\ref{fig:partition-stretch} shows $\lambda_{4}$ and $2*\lambda_{4}$.

\begin{figure}[hb]
\centering
\begin{tikzpicture}[x=4.5mm,y=4.5mm,line width=0.5pt]
\foreach \row/\parts in {0/4,1/3,2/2,3/1}{
  \foreach \col in {1,...,\parts}{
    \draw ({\col-1},-\row) rectangle (\col,{-\row-1});
  }
}
\node at (2,0.75) {\small $\lambda_{4}=(4,3,2,1)$};
\draw[line width=0.75pt,->] (4.7,-2) -- (7.3,-2);
\node at (6,-1.4) {\small $s=2$};
\foreach \row/\parts in {0/4,1/3,2/2,3/1}{
  \foreach \col in {1,...,\parts}{
    \draw ({8+2*(\col-1)},-\row) rectangle ({8+2*\col},{-\row-1});
    \draw[gray!70] ({7+2*\col},-\row) -- ({7+2*\col},{-\row-1});
  }
}
\node at (12,0.75) {\small $2*\lambda_{4}=(8,6,4,2)$};
\end{tikzpicture}
\caption{The staircase partition $\lambda_{4} = (4,3,2,1)$ and its horizontal stretch $2*\lambda_{4} = (8,6,4,2)$. The gray dividers separate the two boxes arising from each original box.}
\label{fig:partition-stretch}
\end{figure}
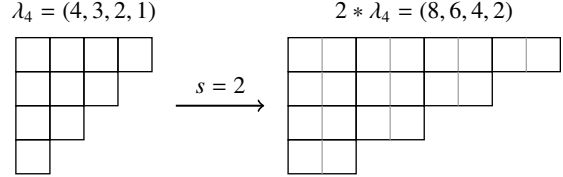

Let $\lambda_{\nu}=(\nu,\nu-1,\dots,1)$. Linz \cite[Theorem~2]{Linz2022} gives the Littlewood--Richardson interpretation
\begin{equation}
\label{eq:LR}
d^{\,(0)}_{2\nu,\sigma} = \mathrm{Cat}^{(s)}_{\nu} = 
c^{\,s*\lambda_{2\nu}}_{\,s*\lambda_{2\nu-1},\,s*(\nu,\nu)}~, \quad s = 2\sigma.
\end{equation}
The proof uses a reflection-principle bijection on Littlewood--Richardson tableaux. For fixed $\nu$, the coefficient in \eqref{eq:LR} is polynomial in $s$, explaining, as Linz notes, the polynomial dependence that Curtright, Van Kortryk, and Zachos \cite{CurtrightVanKortryk2017} observed empirically for the singlet multiplicities. The Littlewood--Richardson tableau model also supplies a combinatorial interpretation of the $s$-Catalan numbers, as Belbachir and Igueroufa \cite{BelbachirIgueroufa2020a} sought. Linz's Theorem~4 further gives, for integers $s \geq 1$ and $\mu \geq 0$,
\begin{equation}
\label{eq:linz4}
\bis{2\mu+1}{s(\mu+1)}{s}-\bis{2\mu+1}{s(\mu+1)+1}{s} = 
\bis{2\mu+2}{s(\mu+1)}{s}-\bis{2\mu+2}{s(\mu+1)+1}{s}.
\end{equation}
At $s=1$, this identity reduces to the two familiar expressions for the Catalan numbers. In the notation of~\eqref{eq:dJ}, with $\sigma = s/2$, its two sides are $d^{\,(\sigma)}_{2\mu+1,\sigma}$ and $d^{\,(0)}_{2\mu+2,\sigma}$. It also follows in two lines from the longitudinal recurrence $\bis{n+1}{K}{s} = \sum_{i=0}^{s}\bis{n}{K-i}{s}$ \cite[eq.~(1.4)]{BelbachirBouroubi2008}. With $n = 2\mu+1$ and $K = s(\mu+1)$, the difference of two such sums telescopes to $\bis{n}{K-s}{s}-\bis{n}{K+1}{s}$, and $\bis{n}{K-s}{s} = \bis{n}{K}{s}$ by~\eqref{eq:symmetry}, since $(K-s)+K = ns$.

Linz's ``spin $\sigma$-Catalan number'' \cite[Def.~3]{Linz2022} uses a shifted index, corresponding to the multiplicity of $V_{\sigma}$ in $V_{\sigma}^{\otimes n}$. Self-duality, $V_{\sigma}^{*} \cong V_{\sigma}$, gives $d^{\,(\sigma)}_{n,\sigma} = d^{\,(0)}_{n+1,\sigma}$, which explains the one-step offset between Linz's table and~\eqref{eq:29} and contains~\eqref{eq:linz4} as the case $n = 2\mu+1$.

\begin{example}[Four-particle singlets]
\label{ex:four-singlets}
For four identical spin-$\sigma$ particles, put $s = 2\sigma \geq 1$. The Clebsch--Gordan rule gives $V_{\sigma} \otimes V_{\sigma} \cong \bigoplus_{J=0}^{s}V_{J}$. A tensor product $V_{J} \otimes V_{L}$ contains a singlet precisely when $J=L$, and then with multiplicity one. Pairing two copies of the two-particle decomposition therefore gives
\[
d^{\,(0)}_{4,s/2}=\mathrm{Cat}^{(s)}_{2}=\bis{4}{2s}{s}-\bis{4}{2s+1}{s}=s+1.
\]
For $s=1,2,3$, the rows $n=4$ in Table~\ref{tab:bisnomial-triangles} give $6-4=2$, $19-16=3$, and $44-40=4$. Equation~\eqref{eq:LR} consequently evaluates the stretching polynomial
\[
c^{\,(4s,3s,2s,s)}_{\,(3s,2s,s),\,(2s,2s)}=s+1.
\]
\end{example}

\subsection{Entanglement witnesses and spectral degeneracies}

In the entanglement witnesses studied by Cohen, Hansen, and Itzhaki \cite{CohenHansen2016}, the singlet multiplicities $d^{\,(0)}_{2\nu,\sigma}$ give the dimension of the witness's singlet eigenspace. For $\sigma = \tfrac{1}{2}$, the spin-$J$ multiplicities give the spectral degeneracies within each admissible fixed-particle-number sector of the symmetric simple exclusion process on the complete graph; over the full state space, the degeneracies are $(2J+1){\mku}d^{\,(J)}_{n,1/2}$. Mendon\c{c}a \cite{Mendonca2013} treats complete bipartite and related graphs.

\subsection{Constant-sum antichains and hash-based signatures}
\label{sec:hash}

The Winternitz one-time signature framework encodes a message as a word $(k_{1},\dots,k_{n})$ over the alphabet $\{0,1,\dots,s\}$. Signing a word reveals, for each $i$, the $k_{i}$th element of the $i$th hash chain, and anyone can advance along a hash chain but not reverse it. A signature for one codeword therefore yields a signature for every componentwise larger word, so the codewords must form an antichain. Every constant-sum layer satisfies this condition because the rank levels of the product of chains in Section~\ref{sec:chains} are antichains.

Such a layer contains $\bis{n}{k}{s}$ words and is largest at a central rank by \eqref{eq:mode}. Zhang, Cui, and Yu \cite{ZhangCui2023} therefore take $k = \lfloor ns/2\rfloor$. They evaluate its size using Belbachir, Bouroubi, and Khelladi's longitudinal recurrence \cite[eq.~(1.4)]{BelbachirBouroubi2008} together with the support-sharp closed form \eqref{eq:BBK21}, and report signature-size reductions of up to $12.4\%$ for SPHINCS$^{+}$ and $1.78\%$ for XMSS. Earlier work on constant-sum encodings in this signature framework is due to Perin, Zambonin, Cust\'odio, Moura, and Panario \cite{PerinZambonin2021}.

\begin{example}[A seven-word constant-sum codebook]
\label{ex:codebook}
For $(n,s)=(3,2)$, the central layer $k=3$ consists of
\[
\{\,012,021,102,111,120,201,210\,\},
\]
where $abc$ denotes the word $(a,b,c)$. Its size is $\bis{3}{3}{\,2}=7$, as in Table~\ref{tab:bisnomial-triangles}, and it is a largest antichain in $\{0,1,2\}^{3}$. If one word were componentwise larger than another, with at least one strict inequality, its coordinate sum would be larger, so two distinct words in this layer are incomparable. For instance, a signature for $012$ permits advancing the first hash chain to obtain a signature for $112$. The latter word has sum four and lies outside the codebook; every componentwise advance that changes a word similarly leaves the constant-sum layer.
\end{example}


\section{Unequal bounds, boundary crossover, and further directions}
\label{sec:extensions}

\subsection{Mixed bounds and mixed spins}

For a collection of particles with spins $\bm{\sigma}=(\sigma_{1},\dots,\sigma_{n})$, put $s_{i}=2\sigma_{i}$ and write $S=\sum_{i}\sigma_{i}$. The dimension of the magnetization-$m$ space is
\begin{equation}
R_{2\bm{\sigma}}(m+S),
\end{equation}
and, with $\ell=S-J$, the spin-$J$ multiplicity is
\begin{equation}
\label{eq:mixedspin}
d^{\,(S-\ell)}_{\bm{\sigma}} = R_{2\bm{\sigma}}(\ell)-R_{2\bm{\sigma}}(\ell-1).
\end{equation}
The spin labels satisfy $0\leq J\leq S$ and $J-S\in\ZZ$, so $\ell=S-J$ is an integer.
The right-hand side of \eqref{eq:mixedspin} also gives the number of symmetric chains beginning at rank $\ell$ in the unequal product $\prod_{i}\{0,1,\dots,2\sigma_{i}\}$. Each such chain has $2(S-\ell)+1$ elements. Thus Proposition~\ref{prop:scd-isotypic} shows that, at the level of profiles, the numbers and lengths of the chains agree with the multiplicities and dimensions of the mixed-spin constituents.

If $n \geq 2$ and $0 \leq \ell \leq \min_{i}(2\sigma_{i})$, the inhomogeneous boundary formula~\eqref{eq:27} and Pascal's rule, $\C{\ell+n-1}{n-1}-\C{\ell+n-2}{n-1} = \C{\ell+n-2}{n-2}$, give
\begin{equation}
\label{eq:mixednearspin}
d^{\,(S-\ell)}_{\bm{\sigma}} = \C{\ell+n-2}{n-2}.
\end{equation}
Throughout this range, the multiplicity depends only on $n$ and $\ell$, not on the individual spins. Beyond this boundary range, the mixed-spin multiplicities still admit a closed form. Multiplying the inhomogeneous generating polynomial by $1-x$ gives
\begin{equation}
\label{eq:Dinhomgf}
\sum_{k \geq 0}\big[R_{\bm{s}}(k)-R_{\bm{s}}(k-1)\big]{\mku}x^{k}
= \frac{\prod\limits_{i=1}^{n}\big(1-x^{s_{i}+1}\big)}{(1-x)^{n-1}},
\end{equation}
and extracting the coefficient of $x^{k}$ yields, for $n \geq 2$,
\begin{equation}
\label{eq:Dinhom}
R_{\bm{s}}(k)-R_{\bm{s}}(k-1) = \sum_{I \subseteq \{1,\dots,n\}} (-1)^{|I|}\,\C{\,k-\sum\limits_{i \in I}(s_{i}+1)+n-2\,}{n-2},
\end{equation}
the inhomogeneous counterpart of \eqref{eq:Dsum}, valid for $k \geq 0$. On the lower half of the rank range, $0 \leq k \leq \lfloor N/2\rfloor$, the left-hand side of \eqref{eq:Dinhom} is exactly the mixed-spin multiplicity $d^{\,(S-k)}_{\bm{\sigma}}$. When all bounds are equal, the $2^{n}$ subset-indexed terms are grouped, according to $j = |I|$, into the terms of~\eqref{eq:Dsum}. Under convention~\eqref{eq:elementary}, every subset $I$ satisfying $\sum_{i \in I}(s_{i}+1) > k$ contributes zero; the Cauchy-product argument of Section~\ref{sec:conventions}, applied to~\eqref{eq:Dinhomgf}, shows that these subsets do not arise at all.

In the range where the first differences give multiplicities, equation~\eqref{eq:Dinhom} is the subset-indexed form of Gyamfi and Barone's formula \cite[Theorem~5]{GyamfiBarone2018}; the boundary formulas~\eqref{eq:nearspin} and~\eqref{eq:mixednearspin} follow by retaining the term indexed by the empty subset.

\begin{example}
\label{ex:mixed}
Take $\bm{\sigma}=(\tfrac{1}{2},1,1)$, so $\bm{s}=(1,2,2)$, $N=5$, and $S=\tfrac{5}{2}$. The rank-size sequence is $(1,3,5,5,3,1)$, whose lower-half first differences are $1,2,2$. Hence the product of chains $\{0,1\} \times \{0,1,2\}^{2}$ decomposes into one six-element symmetric chain beginning at rank $0$, two four-element symmetric chains beginning at rank $1$, and two two-element symmetric chains beginning at rank $2$. The corresponding Clebsch--Gordan decomposition is
\begin{equation}
V_{1/2} \otimes V_{1} \otimes V_{1} \cong V_{5/2} \oplus 2V_{3/2} \oplus 2V_{1/2}.
\end{equation}
This decomposition has the same length--dimension and start--multiplicity profiles as the chain decomposition, and formula~\eqref{eq:Dinhom} returns $1,2,2$ directly.
\end{example}

The convolution interpretation, moment formulas, log-concavity, and existence of a central mode in Section~\ref{sec:prob} extend directly to unequal bounds. Let $X_{i} \sim U_{s_{i}}$ be independent and set $S_{\!n}=X_{1}+\cdots+X_{n}$. Then
\begin{equation}
\label{eq:mixedmoments}
\EE(S_{\!n}) = \frac{N}{2},
\quad
\operatorname{Var}(S_{\!n}) = \sum_{i=1}^{n}\frac{s_{i}(s_{i}+2)}{12} \eqd \tau_{\bm{s}}^{2}.
\end{equation}
Each factor $1+x+\cdots+x^{s_{i}}$ has a log-concave coefficient sequence with no internal zeros. Convolution preserves these properties, so $k \mapsto R_{\bm{s}}(k)$ is log-concave and unimodal. Symmetry places a maximum at $k=\lfloor N/2\rfloor$. Independently, the symmetric chain decomposition of Section~\ref{sec:chains} shows that this central layer is a largest antichain. Consider a triangular array of bound vectors $\bm{s}^{(n)}$, and write $\bm{s}$ for its $n$th row. A Lindeberg central limit theorem applies under the sufficient conditions
\begin{equation}
\tau_{\bm{s}}^{2} \longrightarrow \infty,
\quad
\frac{\max\limits_{1 \leq i \leq n}s_{i}^{2}}{\tau_{\bm{s}}^{2}} 
\longrightarrow 0.
\end{equation}
These conditions hold, in particular, for uniformly bounded arrays with a growing number of nonzero bounds. Since the corresponding probability mass functions are log-concave with no internal zeros, Bender's local-limit lemma \cite[Lemma~2]{Bender1973} upgrades this central limit theorem to
\begin{equation}
\label{eq:mixedllt}
\sup_{k \in \ZZ}\,\Bigg|\tau_{\bm{s}}\frac{R_{\bm{s}}(k)}{\prod\limits_{i=1}^{n}(s_{i}+1)} - \varphi\bigg(\frac{k-N/2}{\tau_{\bm{s}}}\bigg)
\Bigg| \longrightarrow 0.
\end{equation}

\subsection{The boundary crossover}

The boundary formula~\eqref{eq:Dboundary} remains asymptotically valid beyond its exact range. For $k$ indistinguishable balls in $n-1$ urns, Menon and Indira's Poisson limit theorem \cite[Theorem~2]{MenonIndira1990} shows that the number of urns holding exactly $s+1$ balls is asymptotically Poisson when $k \to \infty$, $k/n \to 0$, and $k^{s+1}/n^{s}$ converges. Urns holding more than $s+1$ balls are then asymptotically absent, their expected number being $O(k^{s+2}/n^{s+1})$, so the theorem implies the following crossover.

\begin{proposition}[Boundary crossover]
\label{prop:boundary-crossover}
Fix $s \geq 1$, and let $k=k(n)$ be a sequence of integers tending to infinity. If
\begin{equation}
\frac{k}{n^{s/(s+1)}} \longrightarrow \lambda \in [0,\infty),
\end{equation}
then
\begin{equation}
\label{eq:boundary-crossover}
\frac{D_{n,s}(k)}{\displaystyle\C{k+n-2}{n-2}} 
\longrightarrow e^{-\lambda^{s+1}}.
\end{equation}
Consequently, if $k=o\big(n^{s/(s+1)}\big)$, then $D_{n,s}(k) \sim \C{k+n-2}{n-2}$.
\end{proposition}

\begin{proof}
Put $r=s+1$, $\theta=\lambda^{r}$ and $B_{n}(k)=\C{k+n-2}{n-2}$, and note that $k/n \to 0$. Since $(1-x)(1+x+\cdots+x^{s}) = 1-x^{r}$, the generating function~\eqref{eq:Dgf} equals $(1-x^{r})(1+x+\cdots+x^{s})^{n-1}$, and therefore
\begin{equation}
D_{n,s}(k) = \bis{n-1}{k}{s}-\bis{n-1}{k-r}{s}.
\end{equation}
Choose a weak composition of $k$ into $n-1$ parts uniformly, and let $Y_{n}$ be the number of parts at least $r$. Then $\PP(Y_{n}=0)=\bis{n-1}{k}{s}/B_{n}(k)$. Requiring any prescribed set of $j$ parts to be at least $r$ and subtracting $r$ from each gives a bijection with the weak compositions of $k-rj$, so that $\EE\big[\C{Y_{n}}{j}\big] = \C{n-1}{j}B_{n}(k-rj)/B_{n}(k)$. For each fixed integer $t \geq 0$,
\begin{equation}
\label{eq:Bratio}
\frac{B_{n}(k-t)}{B_{n}(k)} = \prod_{\ell=0}^{t-1}\frac{k-\ell}{n+k-2-\ell} = \bigg(\frac{k}{n}\bigg)^{t}\big(1+o(1)\big),
\end{equation}
because each factor is $(k/n)\big(1+o(1)\big)$ when $k \to \infty$ and $k/n \to 0$. With $\C{n-1}{j} \sim n^{j}/j!$, this gives, for each fixed $j \geq 1$,
\begin{equation}
\EE\bigg[\C{Y_{n}}{j}\bigg] \sim \frac{n^{j}}{j!}\bigg(\frac{k}{n}\bigg)^{rj} = \frac{1}{j!}\bigg(\frac{k^{r}}{n^{r-1}}\bigg)^{j} \longrightarrow \frac{\theta^{j}}{j!},
\end{equation}
since $k^{r}/n^{r-1} = \big(k/n^{s/(s+1)}\big)^{s+1}$. Formula~\eqref{eq:globalsharp} for $n-1$ types, divided by $B_{n}(k)$, reads 
\begin{equation}
\PP(Y_{n}=0) = \sum_{j \geq 0}(-1)^{j}{\mku}\EE\bigg[\C{Y_{n}}{j}\bigg].
\end{equation}
 By the Bonferroni inequalities, its partial sums ending at odd and at even $j$ bound $\PP(Y_{n}=0)$ from below and from above, respectively. Letting $n \to \infty$ with the cutoff fixed, and then letting the cutoff grow, gives $\PP(Y_{n}=0) \to \sum_{j \geq 0}(-1)^{j}\theta^{j}/j! = e^{-\theta}$. The factorial-moment limits also imply that $Y_{n}$ converges in distribution to a Poisson random variable of mean $\theta$. Moreover, by~\eqref{eq:Bratio} with $t = r$,
\begin{equation}
0 \leq \frac{\displaystyle\bis{n-1}{k-r}{s}}{B_{n}(k)} \leq \frac{B_{n}(k-r)}{B_{n}(k)} = O\bigg(\bigg(\frac{k}{n}\bigg)^{r}\bigg) = o(1).
\end{equation}
Dividing the identity $D_{n,s}(k) = \bis{n-1}{k}{s}-\bis{n-1}{k-r}{s}$ by $B_{n}(k)$ and applying the two preceding limits proves~\eqref{eq:boundary-crossover}.
\end{proof}

Thus $n^{s/(s+1)}$ is the first boundary-crossover scale, the smallest scale of $k$ at which~\eqref{eq:Dboundary} acquires a nontrivial limiting correction, namely $e^{-\lambda^{s+1}}$. Figure~\ref{fig:boundary-crossover}
illustrates this convergence for $s=2$.

\begin{figure}[htbp]
\centering
\includegraphics[width=0.65\linewidth]{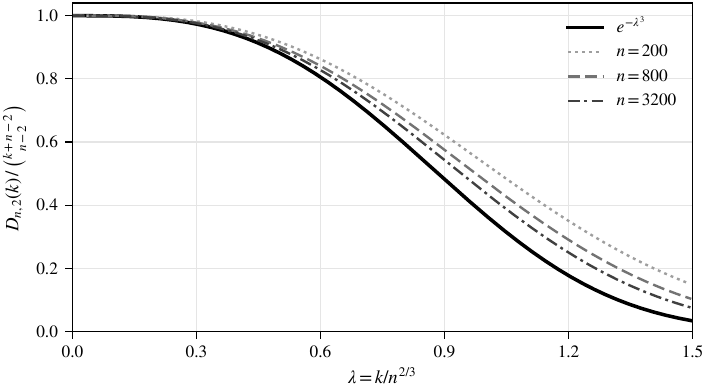}
\caption{Boundary crossover for $s=2$. The finite-$n$ curves plot the exact ratio $D_{n,2}(k)\big/\C{k+n-2}{n-2}$ at every integer $0 \leq k \leq \lfloor 1.5n^{2/3}\rfloor$, against $\lambda = k/n^{2/3}$. The finite-$n$ coefficients were evaluated by exact-integer convolution, followed by first differencing and normalization. The curves approach the limiting profile $e^{-\lambda^{3}}$ of Proposition~\ref{prop:boundary-crossover}. For fixed $\lambda > 0$, the fixed-shift ratios in its proof have relative corrections of order $k/n+1/k$, whose larger scale here is $n^{-1/3}$, suggesting the relatively slow finite-size convergence visible in the figure.}
\label{fig:boundary-crossover}
\end{figure}

The same argument applies when a positive proportion of the bounds attains a fixed minimum.

\begin{corollary}[Unequal-bound boundary crossover]
\label{cor:mixed-boundary-crossover}
Fix $s \geq 1$, and let $\bm{s}=(s_{1},\dots,s_{n})$ depend on $n$, with $s_{i} \geq s$ and $m_{n}=\#\{i\colon s_{i}=s\}$ satisfying $m_{n}/n \to \alpha \in (0,1]$. If $k=k(n) \to \infty$ and $k/n^{s/(s+1)} \to \lambda \in [0,\infty)$, then
\begin{equation}
\label{eq:mixed-boundary-crossover}
\frac{R_{\bm{s}}(k)-R_{\bm{s}}(k-1)}{\displaystyle\C{k+n-2}{n-2}}
\longrightarrow e^{-\alpha\lambda^{s+1}}.
\end{equation}
\end{corollary}

\begin{proof}
For large $n$, permute the coordinates so that $s_{1}=s$, put $r=s+1$ and $B_{n}(k)=\C{k+n-2}{n-2}$, and let $A_{n}(k)$ be the number of bounded compositions in the remaining $n-1$ coordinates. The generating function~\eqref{eq:Dinhomgf} gives
\begin{equation}
\label{eq:mixed-boundary-peel}
R_{\bm{s}}(k)-R_{\bm{s}}(k-1)=A_{n}(k)-A_{n}(k-r).
\end{equation}
Choose an unrestricted weak composition of $k$ into $n-1$ parts uniformly, and let $Y_{n}$ be the number of violations among the $m_{n}-1$ coordinates with bound $s$. By~\eqref{eq:Bratio}, for each fixed $j \geq 1$,
\begin{equation}
\EE\bigg[\C{Y_{n}}{j}\bigg]
=\C{m_{n}-1}{j}\frac{B_{n}(k-rj)}{B_{n}(k)}
\longrightarrow \frac{(\alpha\lambda^{r})^{j}}{j!}.
\end{equation}
The Bonferroni argument of Proposition~\ref{prop:boundary-crossover} yields $\PP(Y_{n}=0) \to e^{-\alpha\lambda^{r}}$. Writing $X_{i}$ for the part in coordinate $i$, the tail probability $\PP(X_{i} \geq t)=B_{n}(k-t)/B_{n}(k)$ is non-increasing in $t$. A violation in a coordinate with bound larger than $s$ requires at least $r+1$ objects, so no upper bound on the larger $s_{i}$ is needed: the union bound shows that the probability of any such violation is at most
\begin{equation}
n\frac{B_{n}(k-r-1)}{B_{n}(k)}
=O\bigg(\frac{k^{r}}{n^{r-1}}\frac{k}{n}\bigg)=o(1),
\end{equation}
since $k^{r}/n^{r-1} \to \lambda^{r}$ and $k/n \to 0$. Consequently,
\begin{equation}
\PP(Y_{n}=0)-o(1) \leq \frac{A_{n}(k)}{B_{n}(k)} \leq \PP(Y_{n}=0).
\end{equation}
Thus $A_{n}(k)/B_{n}(k) \to e^{-\alpha\lambda^{r}}$, whereas $0 \leq A_{n}(k-r)/B_{n}(k) \leq B_{n}(k-r)/B_{n}(k)=o(1)$. Substitution in~\eqref{eq:mixed-boundary-peel} proves the result.
\end{proof}

\subsection{Deformations and arithmetic}

Grouping equal bounds in the inhomogeneous generating polynomial produces $\prod_{i \geq 1}(1+x+\cdots+x^{i})^{L_{i}}$, whose coefficients are the classical supernomial coefficients of Schilling and Warnaar \cite{SchillingWarnaar1998}. Their $q$-supernomials retain recurrences and symmetries while connecting the rank polynomial with partition theory and solvable lattice models. In the homogeneous case, Belbachir and Benmezai \cite{BelbachirBenmezai2014} and Bazeniar, Ahmia, and Belbachir \cite{BazeniarAhmia2018} develop related $q$-deformations and symmetric-function identities. Belbachir and Igueroufa \cite{BelbachirIgueroufa2020b} develop arithmetic properties of generalized binomial coefficients, including Lucas-type congruences and extensions of theorems of Kummer and Ram. Belbachir, Bouroubi, and Khelladi \cite{BelbachirBouroubi2008} relate the homogeneous coefficients to Bell polynomials and Fibonacci numbers.

\subsection{Open problems}
\label{sec:open}

\paragraph{Deformations with a chain interpretation}
Schilling and Warnaar \cite{SchillingWarnaar1999} give crystal-path and generalized Kostka interpretations of $q$-deformed coefficient arrays. For the deformations discussed above, the more specific question is whether a chain model can respect the nested-hook rule of Remark~\ref{rem:cg}. One target is a non-negative integer grading $g(C)$, constant along each symmetric chain, for which the coefficient at rank $k$ is the sum of $q^{g(C)}$ over chains meeting that rank. The question is whether a specified deformation admits such a grading on the coordinate grid, compatible with the nested-hook induction. For this model, the lower-half first differences must be coefficientwise non-negative. Which of these deformations admit this additional structure?

\paragraph{Weighted antichains}
For positive weights $w_{i}$, every level $\sum_{i}w_{i}k_{i} = c$ is an antichain. For homogeneous products with $n \geq 2$ and $s \geq 1$, Griggs's theorem \cite{Griggs1984} implies that a weighted level attains the width of the poset if and only if all weights are equal. The application-specific question is to optimize the level size under additional constraints that a signature scheme imposes on weights or encoding parameters, as in Section~\ref{sec:hash}.

\paragraph{Counting the decompositions}
Proposition~\ref{prop:scd-isotypic} shows that every symmetric chain decomposition of $P_{\bm{s}}$ has the multiset of chain lengths dictated by the Clebsch--Gordan rule. For homogeneous products, Tomon \cite{Tomon2025} proved
\begin{equation}
\#\operatorname{SCD}(\{0,1\}^{n}) = \bigg(\frac{n}{2e}+o(n)\bigg)^{2^{n}},
\end{equation}
and, more generally, uniformly for $s \geq 1$,
\begin{equation}
\#\operatorname{SCD}(\{0,1,\dots,s\}^{n}) = n^{(1-o_{n}(1))(s+1)^{n}}.
\end{equation}
The second estimate determines the logarithmic growth. Further questions concern identifying an analogue of the factor $1/2e$ for fixed $s \geq 2$ and obtaining sharper asymptotic estimates for the number of symmetric chain decompositions of unequal products $P_{\bm{s}}$.


\section{An algebraic coda on conventions for binomial coefficients}
\label{sec:conventions}

The support-sharp formula~\eqref{eq:BBK21} contains only binomial coefficients with non-negative upper arguments and hence does not depend on how we extend ordinary binomial coefficients to negative arguments. The proper-subset sums \eqref{eq:9} and \eqref{eq:11} do depend on that extension, and the elementary convention~\eqref{eq:elementary} is the extension the generating function itself supplies.

To see this, write the rank polynomial as
\begin{equation}
\prod_{i=1}^{n}\big(1-x^{a_{i}}\big)\cdot\frac{1}{(1-x)^{n}},
\quad a_{i} = s_{i}+1 ,
\end{equation}
and expand the $n+1$ factors as a Cauchy product. Each of the first $n$ series has coefficient $1$ at exponent $0$, coefficient $-1$ at exponent $a_{i}$, and no other terms, while the last has coefficient $\C{u+n-1}{n-1}$ at exponent $u$. A nonzero contribution to $[x^{k}]$ therefore selects the subset $I \subseteq [n] = \{1,\dots,n\}$ of indices for which the exponent $a_{i}$ is selected, carries the sign $(-1)^{|I|}$, and leaves the exponent $k-a_{I}$, with $a_{I} = \sum_{i \in I}a_{i}$, for the last factor. That exponent is a power-series index, so it is non-negative by construction: the subsets with $a_{I} > k$ contribute no term at all. They are not annihilated by a convention; they do not arise.

Inclusion--exclusion proceeds differently. Every subset is present from the start, and the terms indexed by $a_{I} > k$ record empty intersections, which the elementary convention sets to zero. Under the polynomial convention, the additional terms change the value of the full transform; the results below identify their total and the proper-subset sum.

Under the polynomial convention mentioned in Section~\ref{sec:pie}, $\C{u}{v}$ is the polynomial $u(u-1)\cdots(u-v+1)/v!$ in $u$, and for negative $u$ this gives
\begin{equation}
\label{eq:25}
\C{u}{v} = (-1)^{v}\C{-u+v-1}{v}.
\end{equation}

\begin{proposition}[The full polynomial transform]
\label{prop:vanish}
Let $n \geq 1$ and write $[n] \deq \{1,\dots,n\}$. Let $\bm{s}=(s_{1},\dots,s_{n})$ have non-negative integer entries, and set $a_{i}=s_{i}+1$, $a_{I} = \sum_{i \in I}a_{i}$, and $N = \sum_{i=1}^{n}s_{i}$. Under the polynomial convention, for every $k \in \ZZ$,
\begin{equation}
\label{eq:vanish}
\sum_{I\subseteq[n]}(-1)^{|I|}\,\C{\,k-a_{I}+n-1\,}{n-1} = 0.
\end{equation}
\end{proposition}

\begin{proof}
Set $f(x)=\C{x+n-1}{n-1}$, a polynomial of degree $n-1$, and define $\nabla_{h}f(x)=f(x)-f(x-h)$. With the shift $T_{h}f(x) = f(x-h)$, we have $\nabla_{h} = 1-T_{h}$, and since the shifts commute and $T_{a}T_{b} = T_{a+b}$, $\prod_{i}(1-T_{a_{i}}) = \sum_{I\subseteq[n]}(-1)^{|I|}T_{a_{I}}$. Hence the left-hand side of~\eqref{eq:vanish} is
\begin{equation}
\bigg(\prod_{i=1}^{n}\nabla_{a_{i}}\bigg)\,f(k).
\end{equation}
Since $a_{i} \geq 1$, each nonzero-step difference lowers the degree of a nonconstant polynomial by one and annihilates a constant. As $f$ has degree $n-1$, applying $n$ such operators annihilates it.
\end{proof}

Proposition~\ref{prop:vanish} is classical. In the equal-step case $a_{1}=\cdots=a_{n}$ it is the basic annihilation property of the finite difference calculus, $\Delta^{n}f = 0$ for $\deg f < n$ \cite[{\S}5.3]{GrahamKnuth1994}. The unequal-step form is the dimension-one case of the annihilation relations of Dahmen and Micchelli \cite{DahmenMicchelli1988}: for a list $X$ of vectors, the operators $\prod_{x \in Y}\nabla_{x}$ indexed by the cocircuits $Y$ of $X$ define the discrete Dahmen--Micchelli space, and in dimension one the single cocircuit is all of $X$. Its polynomial solutions are exactly the polynomials of degree at most $n-1$; the full discrete space can also contain nonpolynomial quasipolynomials.

The Dahmen--Micchelli setting identifies $R_{\bm{s}}$ as the one-dimensional discrete box spline with unit directions and coordinate bounds $0 \leq k_{i} < a_{i}$. Let $B_{\mathrm{seq}}(k)=\C{k+n-1}{n-1}$ for $k \geq 0$, and let $B_{\mathrm{seq}}(k)=0$ for $k < 0$. Applying $\nabla_{a}$ multiplies its generating series by $1-x^{a}$, so
\begin{equation}
R_{\bm{s}} = \bigg(\prod_{i=1}^{n}\nabla_{a_{i}}\bigg){\mku}B_{\mathrm{seq}}.
\end{equation}
Under the polynomial convention, $B_{\mathrm{poly}}(k)=\C{k+n-1}{n-1}$ extends the same values to a polynomial of degree $n-1$ on all integers. Proposition~\ref{prop:vanish} gives $\big(\prod_{i=1}^{n}\nabla_{a_{i}}\big)B_{\mathrm{poly}}=0$; Corollary~\ref{cor:tail} identifies the contribution of the extra terms.

\begin{corollary}[The tail and the proper-subset transform]
\label{cor:tail}
Under the hypotheses and notation of Proposition~\ref{prop:vanish}, for $0 \leq k \leq N$ the polynomial convention gives
\begin{equation}
\label{eq:tail}
\sum_{\substack{I \subseteq [n] \\ a_{I} > k}}(-1)^{|I|}\,\C{\,k-a_{I}+n-1\,}{n-1} = -R_{\bm{s}}(k),
\end{equation}
and
\begin{equation}
\label{eq:propertransform}
\sum_{I\subsetneq[n]}(-1)^{|I|}\,\C{\,k-a_{I}+n-1\,}{n-1} = \C{N-k+n-1}{n-1}.
\end{equation}
The right-hand side of \eqref{eq:propertransform} equals $R_{\bm{s}}(k)$ if and only if $N-k \leq \min(s_{1},\dots,s_{n})$.
\end{corollary}

\begin{proof}
If $a_{I} \leq k$, then $k-a_{I}+n-1 \geq n-1$, so the elementary and polynomial conventions agree on that term. The subsets satisfying $a_{I} \leq k$ therefore index precisely the nonzero terms of the elementary inclusion--exclusion formula~\eqref{eq:9}, whose sum is $R_{\bm{s}}(k)$. Subtracting this sum from the vanishing transform~\eqref{eq:vanish} proves \eqref{eq:tail}.

For \eqref{eq:propertransform}, isolate the term $I=[n]$ in \eqref{eq:vanish}. Since $a_{[n]}=N+n$, the proper-subset sum is
\begin{equation}
-(-1)^{n}\C{k-N-1}{n-1} = \C{N-k+n-1}{n-1}.
\end{equation}
For the second equality, apply~\eqref{eq:25} with $u = k-N-1 < 0$ and $v = n-1$, which gives $\C{k-N-1}{n-1} = (-1)^{n-1}\C{N-k+n-1}{n-1}$, and note that $-(-1)^{n}(-1)^{n-1} = 1$. Proposition~\ref{prop:iff} gives the stated criterion for equality with $R_{\bm{s}}(k)$.
\end{proof}

\begin{example}[Three values from one alternating sum]
\label{ex:conventions}
Take $(n,s,k)=(3,2,2)$, so $N=6$. The proper-subset sum in~\eqref{eq:11} is
\[
\C{4}{2}-3\C{1}{2}+3\C{-2}{2}.
\]
Under the elementary convention it equals $\bis{3}{2}{\,2}=6$. Under the polynomial convention, $\C{-2}{2}=3$, so the same expression equals $6+9=15$. Including the final term gives the full polynomial transform
\[
6+9-\C{-5}{2}=6+9-15=0.
\]
The value $15=\C{6}{2}$ is the unrestricted enumeration for the complementary total $N-k=4$, as in~\eqref{eq:propertransform}; the extra polynomial terms in the full transform sum to $9-15=-6$, as in~\eqref{eq:tail}.
\end{example}

\begin{remark}
\label{rem:reciprocity}
Equation~\eqref{eq:propertransform} is a reciprocity statement. Its proof turns on the upper negation~\eqref{eq:25}, which, when applied to $B_{\mathrm{poly}}$, gives the classical reciprocity for the simplex: under the polynomial convention and for integers $k \geq 1$, $(-1)^{n-1}B_{\mathrm{poly}}(-k)$ is the number of solutions of $k_{1}+\cdots+k_{n}=k$ in strictly positive integers. The reciprocity of the rank polynomial itself is the palindromic symmetry~\eqref{eq:symmetry}, proved in Section~\ref{sec:sym} by the complement bijection; \eqref{eq:propertransform} uses the simplex reciprocity above, applied to the single term indexed by $I=[n]$. In this form, the proper-subset transform is the unrestricted enumeration for the complementary total $N-k$, and by Proposition~\ref{prop:iff} it agrees with the bounded enumeration exactly in the upper boundary regime.
\end{remark}

When all $s_{i}=s$, grouping subsets by their cardinality reduces \eqref{eq:vanish} to the full homogeneous transform. Under the polynomial convention, the same grouping reduces \eqref{eq:propertransform} to
\begin{equation}
\label{eq:homproper}
\sum_{j=0}^{n-1}(-1)^{j}{\mku}\C{n}{j}{\mku}\C{k-j(s+1)+n-1}{n-1}
= \C{ns-k+n-1}{n-1}, \quad 0 \leq k \leq ns.
\end{equation}
Under the elementary convention, the same alternating sum equals $\bis{n}{k}{s}$ by~\eqref{eq:11}. For $(n-1)s \leq k \leq ns$, all of its upper arguments are non-negative, so the two conventions agree term by term and~\eqref{eq:homproper} reduces to~\eqref{eq:24}. The polynomial convention thus extends this upper boundary identity to the whole admissible range, where it no longer computes $\bis{n}{k}{s}$.


\section{Concluding remarks}
\label{sec:conclusion}

The rank polynomial links bounded compositions and rank sizes of products of chains with discrete-uniform sums and $\mathrm{SU}(2)$ tensor-product characters. The complement involution gives coefficient symmetry and rank reversal and transfers the lower-boundary figurate-number identity to the upper boundary. Combined with the support-sharp inclusion--exclusion formula, the involution reduces the alternating sum to one term at either boundary and roughly halves the worst-case number of summands (Corollary~\ref{cor:sharpsym}). At the level of profiles, Proposition~\ref{prop:scd-isotypic} equates the number of chains beginning at each rank with the corresponding irreducible multiplicity and equates chain lengths with constituent dimensions. The de Bruijn--Tengbergen--Kruyswijk induction mirrors the Clebsch--Gordan rule (Remark~\ref{rem:cg}). Both statements hold for unequal bounds and hence for mixed spins.

For equal bounds, convolution gives log-concavity and an exact central mode. The Gaussian local approximation determines the asymptotic size of the central coefficient, while a higher-order local expansion determines the central first difference, which is smaller by a factor of order $n$. Corollary~\ref{cor:central-corrections} gives the explicit $1/n$ corrections to both leading terms. This difference in order underlies the passage from binomial-type coefficients to Catalan-type first differences. At the boundary, Proposition~\ref{prop:boundary-crossover} identifies $n^{s/(s+1)}$ as the first crossover scale for $D_{n,s}(k)$, with limiting correction $e^{-\lambda^{s+1}}$. Corollary~\ref{cor:mixed-boundary-crossover} extends this limit to unequal bounds, with factor $e^{-\alpha\lambda^{s+1}}$ when the fraction of bounds equal to the fixed minimum $s$ tends to $\alpha>0$. A central rank layer is also a largest antichain, which underpins the coding application of Section~\ref{sec:hash}.

Under the polynomial convention, the full inclusion--exclusion transform vanishes (Proposition~\ref{prop:vanish}), and the terms suppressed by the elementary convention sum to $-R_{\bm{s}}(k)$ (Corollary~\ref{cor:tail}). The proper-subset transform returns $\C{N-k+n-1}{n-1}$ instead, so changing the convention replaces the bounded enumeration by the unrestricted enumeration for the complementary total.

Further questions concern compatibility of $q$-deformed chain models with the nested-hook construction, optimization of weighted levels under signature-scheme constraints, and sharper asymptotic estimates for the number of symmetric chain decompositions of homogeneous and unequal products.


\section*{CRediT authorship contribution statement}
\textbf{Hac\`{e}ne Belbachir:} Conceptualization, formal analysis, investigation, validation, writing -- review \& editing.
\textbf{Jos\'{e}~Ricardo G. Mendon\c{c}a}: Conceptualization, formal analysis, investigation, validation, writing -- original draft, writing -- review \& editing.

\section*{Declaration of generative AI and AI-assisted technologies in the writing process}
During the preparation of this manuscript the authors used Anthropic's Claude Opus~5.5 and OpenAI's GPT-6.1-Sol AI models for drafting assistance and algebraic and numerical checking of displayed identities. The authors planned, surveyed, and prepared the majority of the manuscript, reviewed and revised the AI-assisted material, and take full responsibility for the content of the publication.

\section*{Funding}
H.B. acknowledges support from the RECITS Laboratory, USTHB, and DG-RSDT, Algeria. J.R.G.M. thanks FAPESP, Brazil, for partial support through research grant no.~2020/04475-7.

\section*{Data availability}
No datasets were used or generated in this work. The curves in Figure~\ref{fig:boundary-crossover} were computed in Python from the coefficient formulas in the manuscript.

\section*{Declaration of competing interest}
The authors have no relevant financial or non-financial interests to disclose.


\end{document}